\documentclass[11pt]{article}
\usepackage{amsmath,amssymb,amsthm}
\usepackage{enumitem}
\usepackage{xcolor}
\usepackage{authblk}
\usepackage{autobreak}
\usepackage{tikz}
\usetikzlibrary{calc}
\usepackage[margin=1in]{geometry}

\usepackage{caption}
\usepackage{subcaption}

\newtheorem{theorem}{Theorem}[section]
\newtheorem{lemma}{Lemma}[section]
\newtheorem{corollary}{Corollary}[section]
\theoremstyle{definition}
\newtheorem{definition}{Definition}[section]
\newtheorem{claim}{Claim}[section]

\newtheorem{case}{Case}

\newcommand{\flo}[1]{\lfloor #1 \rfloor}

\usepackage{hyperref}
\hypersetup{
	colorlinks,
	linkcolor={red!60!black},
	citecolor={green!60!black},
	urlcolor={blue!60!black}
}
\usepackage[capitalize]{cleveref}

\begin{document}
\title{Three-color online Ramsey numbers $\tilde{r}(P_3,P_3,P_{\ell})$ and $\tilde{r}(P_3, P_3, C_{\ell})$}
\author[1,2]{Hexuan Zhi}
\author[,1,2]{Yanbo Zhang\thanks{Corresponding author.}}
\affil[1]{School of Mathematical Sciences\\ Hebei Normal University\\ Shijiazhuang 050024, China}
\affil[2]{Hebei Research Center of the Basic Discipline Pure Mathematics\\ Shijiazhuang 050024, China}

\date{}
\maketitle
\let\thefootnote\relax\footnotetext{\emph{Email addresses:} {\tt hxzhi.edu@outlook.com} (H. Zhi), {\tt ybzhang@hebtu.edu.cn} (Y. Zhang)}

\begin{quote}
{\bf Abstract:}
For given graphs $G_1, \ldots, G_k$, let $\tilde{r}(G_1, \ldots, G_k)$ denote their online Ramsey number. In an influential paper on the online Ramsey numbers for paths and cycles, Cyman, Dzido, Lapinskas, and Lo (Electron. J. Combin., 2015) determined the exact values of $\tilde{r}(P_3, P_{\ell})$ and $\tilde{r}(P_3, C_{\ell})$. They also conjectured the exact value of $\tilde{r}(P_4, P_{\ell})$ and the limit of $\tilde{r}(P_k, P_{\ell})/\ell$ as $\ell \to \infty$ for $k \ge 5$. The former conjecture was independently confirmed by Bednarska-Bzd\c{e}ga (European J. Combin., 2024) and Y.B.~Zhang and Y.X.~Zhang (arXiv:2302.13640), while the latter was disproved by Mond and Portier (European J. Combin., 2024). In this paper, we extend this line of research to the three-color setting and establish the exact value of $\tilde{r}(P_3,P_3,P_{\ell})$ for $\ell\ge 2$ and $\tilde{r}(P_3, P_3, C_{\ell})$ for $\ell \ge 16$.


\end{quote}

\section{Introduction}

The concept of the size Ramsey number was introduced in 1978 by Erd\H{o}s, Faudree, Rousseau, and Schelp~\cite{Erdos1978}, and has since become one of the central topics in the study of graph Ramsey theory. For graphs $G_1$ and $G_2$, the size Ramsey number $\hat{r}(G_1,G_2)$ is defined as
\[\hat{r}(G_1,G_2)=\min\{|E(G)|\mid G\rightarrow (G_1,G_2)\},\]
where $G\rightarrow (G_1,G_2)$ means that in every red-blue edge-coloring of $G$, there exists either a red subgraph isomorphic to $G_1$ or a blue subgraph isomorphic to $G_2$. When $G_1=G_2$, the notation is abbreviated as $\hat{r}(G_1)$.

Erd\H{o}s, Faudree, Rousseau, and Schelp~\cite{Erdos1978} posed the following question: For the path $P_n$ on $n$ vertices, is it true that $\hat{r}(P_n)/n^{2}\to 0$? Erd\H{o}s~\cite{Erdos1981} further offered a 100-dollar prize for a proof or disproof of the following problem: Is it true that
\[\hat{r}(P_n)/n\to \infty \quad\text{and}\quad\hat{r}(P_n)/n^{2}\to 0?\]
This problem was completely resolved by Beck~\cite{Beck1983}, who showed that $\hat{r}(P_n)<900n$. Through a series of successive improvements, the best known bounds to date are
\[(3.75-o(1))n<\hat{r}(P_n)\le 74n,\]
where the lower bound was established by Bal and DeBiasio~\cite{Bal2022} and the upper bound was obtained by Dudek and Pra{\l}at~\cite{Dudek2017}.

We now reformulate this problem as an online game. Suppose the game is played between two players: Builder and Painter. Builder is allowed to draw edges, while Painter must immediately color each newly drawn edge either red or blue. Builder wins the game if at any point a red copy of $G_1$ or a blue copy of $G_2$ appears. Thus, Builder wins as soon as a graph $G$ satisfying $G \to (G_1,G_2)$ is constructed. However, if Builder is restricted to drawing only one edge per round and Painter must immediately color it, it is possible for Builder to win after constructing far fewer than $|E(G)|$ edges. The minimum number of edges required for Builder to guarantee a win under this constraint is known as the online Ramsey number $\tilde{r}(G_1,G_2)$. Clearly, $\tilde{r}(G_1,G_2)\le \hat{r}(G_1,G_2)$.

The first version of the online Ramsey game was proposed by Beck~\cite{Beck1993}, and the term \emph{online Ramsey number} was coined by Kurek and Ruci\'{n}ski~\cite{Kurek2005}. In general, given simple graphs $G_1,G_2,\cdots,G_t$ and a color set $[t]$ with $t\ge 2$, the game begins with an empty graph on infinitely many vertices. In each round, Builder adds an edge between two nonadjacent vertices, and Painter immediately assigns one of $t$ colors to it. Builder's goal is to force Painter to create a monochromatic copy of $G_i$ in color $i$ for some $i\in [t]$, in as few rounds as possible, while Painter tries to delay this outcome. Assuming both players play optimally, the \emph{online Ramsey number} $\tilde{r}(G_1,G_2,\cdots,G_t)$ is the minimum number of rounds in which Builder can guarantee a win. When $t=2$ and $G_1=G_2$, we abbreviate this as $\tilde{r}(G_1)$.

For the online Ramsey numbers of paths, Grytczuk, Kierstead, and Pra{\l}at~\cite{Grytczuk2008} proved that $\tilde{r}(P_n)=2n-3$ for $2\le n\le 5$, $\tilde{r}(P_6)=10$, and that for $n\ge 7$,
\[
2n-3\le \tilde{r}(P_n)\le 4n-7.
\]
Pra{\l}at~\cite{Pralat2008,Pralat2012} further determined that $\tilde{r}(P_7)=12$, $\tilde{r}(P_8)=15$, and $\tilde{r}(P_9)=17$. Based on these exact values, Conlon~\cite{Conlon2010} posed the question: Is it true that for $n\ge 4$,
\[
\tilde{r}(P_n)\le 5n/2-5?
\]
This conjecture remains open.

For the non-diagonal case, Cyman, Dzido, Lapinskas, and Lo~\cite{Cyman2015} proved that $\tilde{r}(P_3,P_\ell)=\lceil 5(\ell-1)/4 \rceil$ for $\ell\ge 3$. If we replace the path with the cycle $C_{\ell}$ on $\ell$ vertices, the computation becomes more involved. They obtained the result $\tilde{r}(P_3, C_\ell)=\lceil 5\ell/4 \rceil$ for $\ell\ge 5$. In the same paper, two conjectures were proposed. The first conjecture states that $\tilde{r}(P_4,P_\ell)=\left\lceil 7\ell/5 \right\rceil-1$. This conjecture was independently confirmed by Bednarska-Bzd\c{e}ga~\cite{Bednarska2024} and Zhang and Zhang~\cite{Zhang2023}. The second conjecture asserts that $\lim\limits_{\ell\to\infty}\tilde{r}(P_k,P_{\ell+1})/\ell=3/2$ for $k\ge 5$. Bednarska-Bzd\c{e}ga~\cite{Bednarska2024} showed that if $k=o(\ell)$ and $\ell \to \infty$, then $\tilde{r}(P_k,P_\ell)\le (5/3+o(1))\ell$. Subsequently, Mond and Portier~\cite{Mond2024} established that $\tilde{r}(P_k,P_\ell)=(5/3+o(1))\ell$ for $10\le k=o(\ell)$, thereby disproving the second conjecture posed by Cyman et al. Quite recently, Adamska and Adamski~\cite{Adamska2025} also refuted this conjecture for $7\le k\le 9$.

Beyond the online Ramsey numbers of paths, substantial progress has also been made for other sparse graphs. For results on the online Ramsey numbers of cycles, we refer the reader to~\cite{Adamski2024c}. For those concerning paths versus cycles, see~\cite{Adamski2024a,Adamski2024b,Adamski2024c,Dybizbanski2020}. For paths versus stars, see~\cite{Grytczuk2008,Latip2021,Song2025}.

Motivated by the precise determination of $\tilde{r}(P_3, P_\ell)$ and $\tilde{r}(P_3, C_\ell)$ in the two-color setting by Cyman et al.~\cite{Cyman2015}, we naturally extend these problems to the three-color case. Specifically, we study the online Ramsey numbers $\tilde{r}(P_3, P_3, P_\ell)$ and $\tilde{r}(P_3, P_3, C_\ell)$, and determine their exact values for $\ell \ge 2$ and $\ell \ge 16$, respectively.

Our main results are as follows.
\begin{theorem}
	\label{thm:main1}For $\ell\ge 2$, we have
	\[\tilde{r}(P_3,P_3,P_\ell)=\begin{cases}3\ell/2-1, & \text{if}\ \,\, \ell \equiv 0\pmod{4}\ \text{and}\ \ell\ge 8; \\ \left\lfloor 3\ell/2 \right\rfloor, & \text{otherwise}. \end{cases}\]
\end{theorem}

\begin{theorem}
	\label{thm:main2}For $\ell\ge 16$, we have
	\[\tilde{r}(P_3,P_3,C_\ell)=\begin{cases}(3\ell+1)/2, & \text{if} \,\,\ell \equiv 3\pmod{4}; \\ \left\lfloor 3(\ell+1)/2 \right\rfloor, & \text{otherwise}. \end{cases}\]
\end{theorem}

We will present the lower bounds for \cref{thm:main1} and \cref{thm:main2} in Section~\ref{sec2}. The proofs of the upper bounds for \cref{thm:main1} and \cref{thm:main2} are given in Section~\ref{sec3} and Section~\ref{sec4}, respectively.

\section{The lower bound of \cref{thm:main1} and \cref{thm:main2}}\label{sec2}

When $\ell=2$, Painter can simply color the first two edges red and blue, respectively. This yields $\tilde{r}(P_3,P_3,P_2)\ge 3$. When $\ell=3$, Painter can color the first three edges red, blue, and green, respectively, which gives $\tilde{r}(P_3,P_3,P_3)\ge 4$.

For $\ell\ge 4$, we assume that the graph $G$ is a connected graph with $\ell-1$ edges and maximum degree $2$. By definition, $G$ must be either $P_{\ell}$ or $C_{\ell-1}$. We now prove that
\[
\tilde{r}(P_3,P_3,G)\ge \begin{cases}3\ell/2-1, & \text{if}\ \,\, \ell \equiv 0\pmod{4}\ \text{and}\ \ell\ge 8; \\ \left\lfloor 3\ell/2 \right\rfloor, & \text{otherwise}. \end{cases}
\]
This establishes the lower bounds in \cref{thm:main1} and \cref{thm:main2}.

When $\ell=4$, we prove that in $5$ rounds Builder cannot force a red $P_3$, a blue $P_3$, or a green $G$. Painter colors the first appeared edge red. If the second edge is nonadjacent to the first one, Painter also colors it red. In the next three moves Painter colors one edge blue and two edges green. Then we have the lower bound. Thus we assume that the first two edges are adjacent. Painter colors the second one blue. Let $v_1v_2$ denote the red edge and $v_2v_3$ the blue one. In the next three moves Painter uses the following strategy: if there is an edge nonadjacent to $v_1v_2$, Painter colors the first such edge red; if there is an edge adjacent to $v_1v_2$ but nonadjacent to $v_2v_3$, Painter colors the first such edge blue; Painter colors all other edges green. If there is a red edge or a blue edge in these $3$ rounds, there are at most two green edges and our proof is done. If there are only green edges in these $3$ rounds, these edges are either $v_1v_3$ or incident to $v_2$, which cannot form a green $G$. This completes the case $\ell=4$.

For $\ell\ge 5$, we set $N=3\ell/2-1$ if $\ell\equiv 0\pmod{4}$, otherwise $N=\left\lfloor 3\ell/2 \right\rfloor$. We prove that Builder cannot force a red $P_3$, a blue $P_3$, or a green $G$ in $N-1$ rounds. We split the argument into two cases.

\begin{case}\label{case1}
	$\ell\not\equiv1\pmod{4}$.
\end{case}

For each newly added edge during the game, the order of priority for Painter to color is red, blue, and green. More concretely, let $e_i$ be the edge chosen by Builder in his $i$th move. If $e_i$ is nonadjacent to any red edge, Painter colors $e_i$ red; if $e_i$ is adjacent to a red edge but nonadjacent to any blue one, Painter colors $e_i$ blue; if $e_i$ is adjacent to both red and blue edges, Painter colors $e_i$ green. From Painter's strategy we see that all red edges form a matching and so are the blue edges.

After $N-1$ rounds, let $H$ be the graph induced by all red and blue edges. Then $\Delta(H)\le 2$ and hence $H$ consists of a disjoint union of cycles and paths. Moreover, each cycle in $H$ is an even cycle. Let $X$ be the set of vertices with degree $2$ in $H$. Let $Y$ and $Z$ be the sets of vertices with degree $1$ that are incident to a red edge and a blue edge in $H$, respectively. Thus, $|V(H)|=|X|+|Y|+|Z|$ and $|E(H)|=|X|+(|Y|+|Z|)/2$. Furthermore, we have the following claim.
\begin{claim}\label{clm:even}
Both $|Y|+|Z|$ and $|X|+|Y|$ are even.	
\end{claim}

\begin{proof}
	Since the degree sum of $H$ is even, $|Y|+|Z|$ must be even. To show $|X|+|Y|$ is even, we need only to prove that $|C\cap X|+|C\cap Y|$ is even for every component $C$ of $H$. If $C$ is an even cycle, then $C\cap X=C$ and $C\cap Y=\emptyset$. If $C$ is a path with odd vertices, then $|C\cap X|=|C|-2$ and $|C\cap Y|=1$. If $C$ is a path with even vertices, then $|C\cap X|=|C|-2$ and $|C\cap Y|=0$ or $2$. In all three cases we conclude that $|C\cap X|+|C\cap Y|$ is even. Hence $|X|+|Y|$ is even.
\end{proof}

Suppose to the contrary that there is a green $G$ in $N-1$ rounds. From $|E(H)|+|E(G)|\le N-1$ and $|E(G)|=\ell-1$ we have\begin{equation}|X|+(|Y|+|Z|)/2\le N-\ell.\label{equ1}\end{equation}

Notice that Painter will color an edge green only when it has already been adjacent to a red edge and a blue edge. So each green edge is either incident to $X$, or has one end vertex in $Y$ and the other in $Z$. Let $G_1$ be the graph induced by all edges of the green $G$ that are incident to $X$. Then $G_1$ has at most $2|X|$ edges. Let $G_2$ be the graph induced by all edges of the green $G$ that has both end vertices in $Y\cup Z$.

According to our assumption, a green copy of $G$ has already appeared. Based on Painter's coloring strategy, there must be at least one blue edge in the graph. Once a blue edge appears, it must be adjacent to some red edge. Hence, we conclude that $|X|\ge 1$. Substituting this into Inequality (\ref{equ1}) yields
\begin{equation}|Y|+|Z|\le 2N-2\ell-2\le \ell-2.\label{equ3}\end{equation}

If $|Y|<|Z|$, it follows from Claim~\ref{clm:even} that $|Y|\le |Z|-2$. Since $Y$ is a vertex cover of $G_2$, the graph $G_2$ has at most $2|Y|$ edges. Adding up the edge numbers of $G_1$ and $G_2$, we have $\ell-1=|E(G_1)|+|E(G_2)|\le 2|X|+|Y|+|Z|-2$. Combining it with Inequality (\ref{equ1}) yields $\ell-1\le 2N-2\ell-2\le \ell-2$, a contradiction. If $|Y|>|Z|$, a contradiction can similarly be derived by symmetry. Therefore, we assume that $|Y|=|Z|$.

If $|Y|=|Z|=0$, then $|X|$ is even, and $\ell-1=|E(G_1)|\le 2|X|$. If $\ell\equiv0 \pmod{4}$ and $\ell\ge 5$, it follows from $2|X|\equiv0 \pmod{4}$ that $\ell\le 2|X|$. Combining it with Inequality (\ref{equ1}) yields $\ell\le 2N-2\ell\le \ell-2$, a contradiction. If $\ell\equiv2\ \text{or}\ 3 \pmod{4}$, it follows from $2|X|\equiv0 \pmod{4}$ that $\ell\le 2|X|-1$, which together with Inequality (\ref{equ1}) gives $\ell\le 2N-2\ell-1\le \ell-1$, again a contradiction.

If $|Y|=|Z|>0$, then since $Y\cup Z$ contains $2|Y|$ vertices, the graph $G_2$ has at most $2|Y|$ edges. According to Inequality (\ref{equ3}), the graph $G_1$ must contain at least one edge. Since $G_1$ and $G_2$ form a decomposition of the graph $G$, the graph $G_2$ is either a path or a union of disjoint paths. Therefore, $G_2$ has at most $2|Y|-1$ edges. Adding the numbers of edges in $G_1$ and $G_2$ yields $\ell-1\le 2|X|+2|Y|-1$, which implies $\ell\le 2|X|+2|Y|$.

If $\ell\equiv 0 \pmod{4}$ and $\ell\ge 5$, then combining this with Inequality (\ref{equ1}) leads to a contradiction. By Claim~\ref{clm:even}, we have $(2|X|+2|Y|)\equiv 0 \pmod{4}$. Hence, if $\ell\not\equiv 0 \pmod{4}$, the inequality $\ell\le 2|X|+2|Y|$ implies $\ell\le 2|X|+2|Y|-1$. Combining this with Inequality (\ref{equ1}) again yields a contradiction.

\begin{case}
	$\ell\equiv1 \pmod{4}$ and $\ell\ge 2$.
\end{case}

In this case $N=(3\ell-1)/2$. Painter uses a similar strategy as in Case~\ref{case1}. Let $C_k$ be a cycle of order $k$, each of whose edges is colored red or blue. During the online Ramsey game, the precedence order of colored graphs for Painter to avoid is a non-green $C_k$ for $3\le k\le (\ell-1)/2$, a red $P_3$, and a blue $P_3$. More specifically, let $H_i$ be the graph induced by all red and blue edges before the $i$th round, and let $e_i$ be the edge chosen by Builder in his $i$th move. If adding $e_i$ to $H_i$ would create a cycle of length at most $(\ell-1)/2$, Painter colors $e_i$ green. Otherwise, if $e_i$ is nonadjacent to any red edge, Painter colors $e_i$ red; if $e_i$ is adjacent to a red edge but nonadjacent to any blue one, Painter colors $e_i$ blue; if $e_i$ is adjacent to both red and blue edges, Painter colors $e_i$ green. From Painter's strategy we see that all red edges form a matching and so are the blue edges.

Set $H=H_N$. Suppose by contradiction that a green $G$ appears in $N-1$ rounds. Then $|E(H)|\le (N-1)-(\ell-1)=(\ell-1)/2$. Hence $H$ contains no cycle with length at least $(\ell+1)/2$. By Painter's strategy, $H$ consists of a disjoint union of paths. Let $V_1$ and $V_2$ be the sets of vertices that are incident to a red edge and a blue edge, respectively. Set $X=V_1\cap V_2$, $Y=V_1\setminus V_2$, and $Z=V_2\setminus V_1$. So the notations $X, Y, Z$ are the same as these in Case~\ref{case1}. From $|E(H)|=|X|+(|Y|+|Z|)/2$ we see that \begin{equation}|X|+(|Y|+|Z|)/2\le (\ell-1)/2.\label{equ2}\end{equation}

By Painter's strategy, there must be at least one blue edge in the graph. Moreover, any blue edge must be adjacent to some red edge. Hence, we conclude that $|X|\ge 1$. Substituting into Inequality (\ref{equ2}) yields
\begin{equation}|Y|+|Z|\le \ell-3.\label{equ4}\end{equation}

By Painter's strategy, every green edge is either incident to $X$, or has both endpoints in $Y\cup Z$. Let $G_1$ be the subgraph induced by all edges of the green copy of $G$ that are incident to $X$. Then $G_1$ has at most $2|X|$ edges. Let $G_2$ be the subgraph induced by all green edges with both endpoints in $Y\cup Z$. From Inequality (\ref{equ3}), it follows that $G_1$ contains at least two edges. Since $G_1$ and $G_2$ form a decomposition of the graph $G$, the graph $G_2$ must be a path or a union of disjoint paths. Thus, $G_2$ has at most $|Y|+|Z|-1$ edges. Summing the numbers of edges in $G_1$ and $G_2$, we obtain $\ell-1=|E(G_1)|+|E(G_2)|\le 2|X|+|Y|+|Z|-1$. Combining this with Inequality (\ref{equ2}), we get $\ell\le \ell-1$, a contradiction, which completes this case.\qed

\section{The upper bound of \cref{thm:main1}}\label{sec3}

For $\ell\ge 2$, we define the number of rounds $N$ as follows: $N=3\ell/2-1$ when $\ell\equiv 0\pmod{4}$ and $\ell\ge 8$, and $N=\flo{3\ell/2}$ otherwise. We will show that Builder can always force a red $P_3$, a blue $P_3$, or a green $P_\ell$ within $N$ rounds. Let us assume that neither a red $P_3$ nor a blue $P_3$ appears in the graph. We will then show that Builder can construct a green $P_\ell$ within $N$ rounds.
\begin{lemma}\label{basiclemma}
	In $4$ rounds, Builder can always construct one of the six graphs depicted below.
\end{lemma}

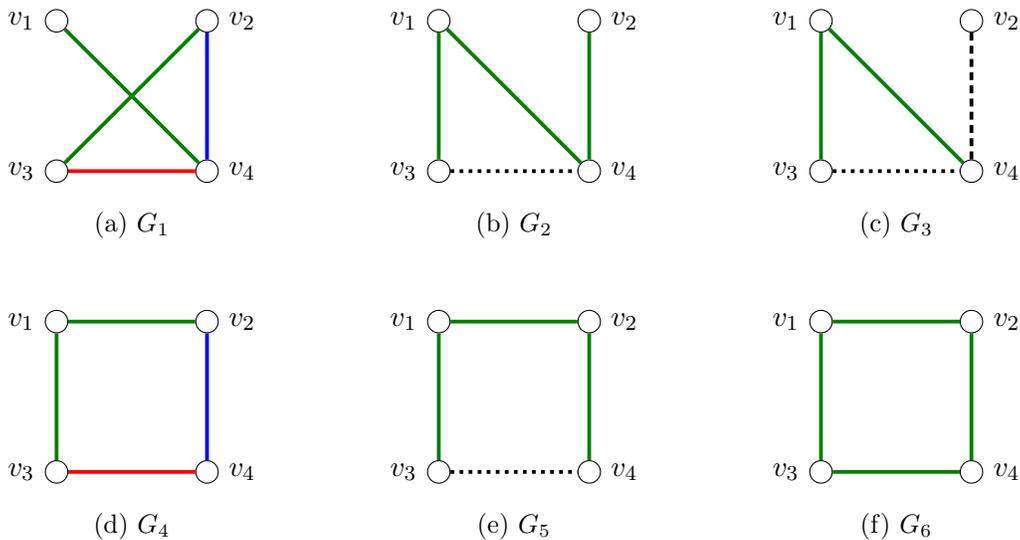
\begin{figure}[ht]
    \centering
	\definecolor{darkgreen}{rgb}{0.0, 0.5, 0.0} 
    \begin{subfigure}{0.3\textwidth}
        \centering
        \begin{tikzpicture}[scale=2, 
            vertex/.style={circle, draw, fill=white, inner sep=3pt},
            every label/.style={scale=1, black}]
        
        \node[vertex, label=left:$v_1$] (v1) at (0, 1) {};
        \node[vertex, label=right:$v_2$] (v2) at (1, 1) {};
        \node[vertex, label=left:$v_3$] (v3) at (0, 0) {};
        \node[vertex, label=right:$v_4$] (v4) at (1, 0) {};
        
        \draw[line width=0.5mm, darkgreen] (v1) -- (v4);
        \draw[line width=0.5mm, darkgreen] (v2) -- (v3);
        \draw[line width=0.5mm, blue] (v2) -- (v4);
        \draw[line width=0.5mm, red] (v3) -- (v4);
	\end{tikzpicture}
	\caption{$G_1$}
	\end{subfigure}
	\begin{subfigure}{0.3\textwidth}
		\centering
	\begin{tikzpicture}[scale=2, 
            vertex/.style={circle, draw, fill=white, inner sep=3pt},
            every label/.style={scale=1, black}]
        
        \node[vertex, label=left:$v_1$] (v1) at (0, 1) {};
        \node[vertex, label=right:$v_2$] (v2) at (1, 1) {};
        \node[vertex, label=left:$v_3$] (v3) at (0, 0) {};
        \node[vertex, label=right:$v_4$] (v4) at (1, 0) {};
        
        \draw[line width=0.5mm, darkgreen] (v1) -- (v3);
        \draw[line width=0.5mm, darkgreen] (v2) -- (v4);
        \draw[line width=0.5mm, dotted] (v3) -- (v4);
		\draw[line width=0.5mm, darkgreen] (v1) -- (v4);
        \end{tikzpicture}
		\caption{$G_2$}
		\end{subfigure}
		\begin{subfigure}{0.3\textwidth}
			\centering
        \begin{tikzpicture}[scale=2, 
            vertex/.style={circle, draw, fill=white, inner sep=3pt},
            every label/.style={scale=1, black}]
        
        \node[vertex, label=left:$v_1$] (v1) at (0, 1) {};
        \node[vertex, label=right:$v_2$] (v2) at (1, 1) {};
        \node[vertex, label=left:$v_3$] (v3) at (0, 0) {};
        \node[vertex, label=right:$v_4$] (v4) at (1, 0) {};
        
        \draw[line width=0.5mm, darkgreen] (v1) -- (v4);
        \draw[line width=0.5mm, darkgreen] (v1) -- (v3);
        \draw[line width=0.5mm, dotted] (v3) -- (v4);
        \draw[line width=0.5mm, densely dashed] (v2) -- (v4);
	\end{tikzpicture}
	\caption{$G_3$}
	\end{subfigure}

	\vspace{2em}

	\begin{subfigure}{0.3\textwidth}
		\centering
		\begin{tikzpicture}[scale=2, 
            vertex/.style={circle, draw, fill=white, inner sep=3pt},
            every label/.style={scale=1, black}]
        
        \node[vertex, label=left:$v_1$] (v1) at (0, 1) {};
        \node[vertex, label=right:$v_2$] (v2) at (1, 1) {};
        \node[vertex, label=left:$v_3$] (v3) at (0, 0) {};
        \node[vertex, label=right:$v_4$] (v4) at (1, 0) {};
        
        \draw[line width=0.5mm, darkgreen] (v1) -- (v2);
        \draw[line width=0.5mm, darkgreen] (v1) -- (v3);
        \draw[line width=0.5mm, blue] (v2) -- (v4);
        \draw[line width=0.5mm, red] (v3) -- (v4);
	\end{tikzpicture}
	\caption{$G_4$}
	\end{subfigure}
	\begin{subfigure}{0.3\textwidth}
		\centering
		\begin{tikzpicture}[scale=2, 
            vertex/.style={circle, draw, fill=white, inner sep=3pt},
            every label/.style={scale=1, black}]
        
        \node[vertex, label=left:$v_1$] (v1) at (0, 1) {};
        \node[vertex, label=right:$v_2$] (v2) at (1, 1) {};
        \node[vertex, label=left:$v_3$] (v3) at (0, 0) {};
        \node[vertex, label=right:$v_4$] (v4) at (1, 0) {};
        
        \draw[line width=0.5mm, darkgreen] (v1) -- (v2);
        \draw[line width=0.5mm, darkgreen] (v1) -- (v3);
        \draw[line width=0.5mm, darkgreen] (v2) -- (v4);
        \draw[line width=0.5mm, dotted] (v3) -- (v4);
	\end{tikzpicture}
	\caption{$G_5$}
	\end{subfigure}
	\begin{subfigure}{0.3\textwidth}
		\centering
	\begin{tikzpicture}[scale=2, 
            vertex/.style={circle, draw, fill=white, inner sep=3pt},
            every label/.style={scale=1, black}]
        
        \node[vertex, label=left:$v_1$] (v1) at (0, 1) {};
        \node[vertex, label=right:$v_2$] (v2) at (1, 1) {};
        \node[vertex, label=left:$v_3$] (v3) at (0, 0) {};
        \node[vertex, label=right:$v_4$] (v4) at (1, 0) {};
        
        \draw[line width=0.5mm, darkgreen] (v1) -- (v2);
        \draw[line width=0.5mm, darkgreen] (v1) -- (v3);
        \draw[line width=0.5mm, darkgreen] (v2) -- (v4);
        \draw[line width=0.5mm, darkgreen] (v3) -- (v4);
        \end{tikzpicture}
		\caption{$G_6$}
	\end{subfigure}
    \caption{Units that Builder can create in 4 rounds, where $\{\text{dotted}, \text{dashed}\} = \{\text{red}, \text{blue}\}$.}
    \label{fig:unit}
\end{figure}

\begin{proof} 
	Note that in graphs $G_2$ and $G_5$, dashed edges represent edges that may be either red or blue. In graph $G_3$, the two dashed edges represent one red edge and one blue edge, respectively. Also observe that the vertex labels in the diagrams are introduced solely for the convenience of later proofs and are irrelevant to the order in which the edges are added.

	Builder first constructs a $P_3$. The colors of the two edges in this $P_3$ fall into one of the following three cases: one red and one blue, exactly one green edge, or two green edges.
	\begin{enumerate}[label=(\arabic*).]
		\item If the two edges of the $P_3$ are one red and one blue, Builder connects the two endpoints of the $P_3$ with a new edge and adds another edge from the center vertex of the $P_3$ to a new vertex. Both of these new edges must be green, yielding graph $G_1$.
		\item If the $P_3$ contains exactly one green edge, Builder connects the two endpoints of the $P_3$ with a new edge $e_1$, and adds another edge $e_2$ from the vertex incident to the non-green edge to a new vertex.
		\begin{enumerate}
			\item If $e_1$ is not green, then $e_2$ must be green, and we obtain graph $G_1$ again.
			\item If both $e_1$ and $e_2$ are green, then we obtain graph $G_2$.
			\item If $e_1$ is green and $e_2$ is not green, then we obtain graph $G_3$.
		\end{enumerate}  
		\item If both edges of the $P_3$ are green, Builder adds two edges from a new vertex to the two endpoints of the $P_3$. The colors of these two new edges fall into exactly three cases: one red and one blue, exactly one green edge, or two green edges. These cases correspond to graphs $G_4$, $G_5$, and $G_6$, respectively.
	\end{enumerate}
	Therefore, within $4$ rounds, Builder can always construct one of the six graphs.
\end{proof}

We now introduce two important definitions. Each graph in Figure~\ref{fig:unit} is referred to as a \emph{unit}. A green path $P_n$ is called \emph{good} if at least one of its endpoints is incident to a red or blue edge; otherwise, it is called \emph{bad}.

\begin{lemma}\label{mainlemma}
	Suppose there is a green path $P_n$ with $n\ge 2$, along with a unit that is disjoint from it. Then the following holds:
	\begin{enumerate}
		\item If the path is good, then either a bad $P_{n+4}$ can be obtained in the next round, or a good $P_{n+4}$ can be obtained within the next $2$ rounds.
		\item If the path is bad, then either a bad $P_{n+4}$ can be obtained in the next round, or a good $P_{n+4}$ can be obtained within the next $3$ rounds.
	\end{enumerate} 
\end{lemma}

\begin{proof}
	Let $u_1$ and $u_2$ denote the two endpoints of the green path $P_n$. If $P_n$ is good, we may assume without loss of generality that $u_1$ is incident to a red or blue edge. We now distinguish several cases based on the structure of the unit.

	\begin{enumerate}[label=(\arabic*).]
		\item If the unit is graph $G_3$, Builder draws edges $v_2v_3$ and $v_4u_1$ in the next $2$ rounds. These two edges must be green. Thus, within $2$ rounds, we obtain a good $P_{n+4}$: $v_2v_3v_1v_4u_1P_nu_2$.
		
		\item If the unit is graph $G_2$ or $G_5$, Builder draws the edge $v_3u_1$. If this edge is green, then a $P_{n+4}$ is formed in $1$ round. If $v_3u_1$ is not green, Builder proceeds to draw $v_3u_2$, which has to be green. It is easy to verify that a good $P_{n+4}$ is obtained within $2$ rounds.
		
		\item Suppose the unit is graph $G_1$ or $G_4$. For a good $P_n$, Builder draws the edge $v_4u_2$. If $u_1$ is incident to a red edge, Builder draws $v_2u_1$; otherwise, since $u_1$ must be incident to a blue edge, Builder draws $v_3u_1$. In either case, the two added edges must be green. Hence, a good $P_{n+4}$ is obtained within $2$ rounds.
		
		For a bad $P_n$, Builder first draws $v_3u_1$. If this edge is green, Builder proceeds with $v_4u_2$, which must also be green. In this case, a good $P_{n+4}$ is obtained in $2$ rounds. If $v_3u_1$ is not green, then it must be blue. Builder then draws both $v_3u_2$ and $v_4u_1$, which must be green. Thus, a good $P_{n+4}$ is obtained within $3$ rounds.
		
		\item Suppose the unit is graph $G_6$. Builder draws the edge $v_1u_1$. If it is green, we obtain a $P_{n+4}$ in $1$ round: $v_2v_4v_3v_1u_1P_nu_2$. Suppose $v_1u_1$ is not green.
		
		For a good $P_n$, Builder draws $v_3u_1$, which must be green. Thus, within $2$ rounds, we obtain a good $P_{n+4}$: $v_1v_2v_4v_3u_1P_nu_2$.
		
		For a bad $P_n$, Builder draws both $v_2u_1$ and $v_3u_1$. At least one of these edges must be green. Hence, a good $P_{n+4}$ is obtained within $3$ rounds: either $v_1v_3v_4v_2u_1P_nu_2$ or $v_1v_2v_4v_3u_1P_nu_2$.
	\end{enumerate}
	This completes the proof of the lemma.
\end{proof}

\begin{corollary}\label{maincor} 
	Suppose there exists a green path $P_n$ with $n\ge 2$. Then the following statements hold:
	\begin{enumerate}
		\item If this is a good $P_n$, then either a bad $P_{n+4}$ can be obtained within the next $5$ rounds, or a good $P_{n+4}$ can be obtained within the next $6$ rounds.
		\item If this is a bad $P_n$, then either a bad $P_{n+4}$ can be obtained within the next $5$ rounds, or a good $P_{n+4}$ can be obtained within the next $7$ rounds.
	\end{enumerate} 
\end{corollary}
\begin{proof}
	By Lemma~\ref{basiclemma}, Builder can construct one of the units in Figure~\ref{fig:unit} within $4$ rounds, such that it is vertex-disjoint from the existing green path $P_n$. Then, by applying Lemma~\ref{mainlemma}, the corollary follows immediately.
\end{proof}

\begin{lemma}\label{P8lemma} 
	Either Builder can create a bad $P_8$ within $10$ rounds, or a good $P_8$ within $11$ rounds.
\end{lemma}
\begin{proof}
	By Lemma~\ref{basiclemma}, Builder can construct two vertex-disjoint units within $8$ rounds.

	If one of the units is isomorphic to $G_2$, $G_5$, or $G_6$, note that each of these contains a green $P_4$ as a subgraph. Then, by Lemma~\ref{mainlemma}, Builder can either create a bad $P_8$ in the next round, or a good $P_8$ within the next $3$ rounds.

	If one of the units is isomorphic to $G_3$, Builder connects vertices $v_2$ and $v_3$. This edge must be green, which yields a good $P_4$: $v_2v_3v_1v_4$. By Lemma~\ref{mainlemma}, Builder can then create a bad $P_8$ in the next round, or a good $P_8$ within the next $2$ rounds.

	If both units are copies of $G_1$, let one of the copies be as shown in Figure~\ref{fig:unit} (a), and label the corresponding vertices in the other copy as $v_1',v_2',v_3',v_4'$. Then the path $v_1v_4v_1'v_4'v_2v_3v_2'v_3'$ forms a good $P_8$.

	If both units are copies of $G_4$, let one of the copies be as shown in Figure~\ref{fig:unit} (d), and label the corresponding vertices in the other copy as $v_1',v_2',v_3',v_4'$. Then the path $v_2v_1v_3v_4'v_4v_2'v_1'v_3'$ forms a good $P_8$.

	If the two units are $G_1$ and $G_4$, with $G_1$ as shown in Figure~\ref{fig:unit} (a), and the corresponding vertices in $G_4$ labeled as $v_1',v_2',v_3',v_4'$, then the path $v_2v_3v_4'v_1v_4v_2'v_1'v_3'$ forms a good $P_8$.
\end{proof}

\begin{theorem}\label{finallemma} 
	For $k\ge 2$, either a bad $P_{4k}$ can be obtained within $6k-2$ rounds, or a good $P_{4k}$ can be obtained within $6k-1$ rounds.
\end{theorem}
\begin{proof}
	We proceed by induction on $k$. For the base case $k=2$, the statement holds by Lemma~\ref{P8lemma}. For $k\ge 3$, assume the statement holds for $k-1$, that is, either a bad $P_{4k-4}$ can be obtained within $6k-8$ rounds, or a good $P_{4k-4}$ can be obtained within $6k-7$ rounds. The desired conclusion then follows by applying Corollary~\ref{maincor}.
\end{proof}

\begin{lemma}\label{P5lemma}
	Either Builder can create a bad $P_5$ within $6$ rounds, or a good $P_5$ within $7$ rounds.
\end{lemma}
\begin{proof}
	By Lemma~\ref{basiclemma}, Builder can construct a unit within $4$ rounds, with vertex labels as shown in Figure~\ref{fig:unit}.

	If the unit is isomorphic to $G_1$, $G_3$, or $G_4$, then Builder adds two edges from vertex $v_2$ to new vertices, at least one of which must be green. Without loss of generality, suppose $v_2u_0$ is a green edge. Then Builder adds the edge $u_0v_4$, which must be green. It is easy to verify that this yields a good $P_5$ within a total of $7$ rounds.

	If the unit is isomorphic to $G_2$ or $G_5$, then Builder adds two edges from vertex $v_3$, at least one of which must be green. In this case, a $P_5$ is obtained within a total of $6$ rounds.

	If the unit is isomorphic to $G_6$, then Builder adds edges from a new vertex $u_0$ to both $v_3$ and $v_4$. If at least one of these is green, a $P_5$ is obtained within $6$ rounds. If neither $u_0v_3$ nor $u_0v_4$ is green, then Builder adds the edge $u_0v_1$, which must be green. In this case, a good $P_5$ is created within $7$ rounds: $u_0v_1v_2v_4v_3$.
\end{proof}

\begin{theorem}\label{final1lemma} 
	For $k\ge 2$, either a bad $P_{4k-3}$ can be obtained within $6k-6$ rounds, or a good $P_{4k-3}$ can be obtained within $6k-5$ rounds.
\end{theorem}
\begin{proof}
	We proceed by induction on $k$. When $k=2$, the statement holds by Lemma~\ref{P5lemma}. For $k\ge 3$, assume the statement holds for $k-1$, i.e., either a bad $P_{4k-3}$ can be obtained within $6k-12$ rounds, or a good $P_{4k-3}$ can be obtained within $6k-11$ rounds. The desired conclusion then follows by applying Corollary~\ref{maincor}.
\end{proof}

\begin{lemma}\label{P7lemma}
	Either Builder can create a bad $P_7$ within $9$ rounds, or a good $P_7$ within $10$ rounds.
\end{lemma}
\begin{proof}
	Builder first creates a star $K_{1,3}$ in $3$ rounds. If at least two of the edges are green, then a bad $P_3$ is formed within $3$ rounds. Applying Corollary~\ref{maincor} then completes the proof. Thus, it suffices to consider the remaining case where the $K_{1,3}$ contains at most one green edge. In order to avoid a red or blue $P_3$, the three edges of the $K_{1,3}$ must be red, blue, and green, respectively. Without loss of generality, assume $u_0u_1$ is green, $u_0u_2$ is red, and $u_0u_3$ is blue.

	By Lemma~\ref{basiclemma}, Builder can construct a unit in $4$ rounds that is vertex-disjoint from the $K_{1,3}$, with vertex labels as shown in Figure~\ref{fig:unit}.

	If the unit is $G_1$, Builder then draws the edges $u_0v_1$, $u_2v_2$, and $u_2v_4$, all of which must be green. This results in a good $P_7$ within $10$ rounds: $v_3v_2u_2v_4v_1u_0u_1$.

	If the unit is $G_4$, Builder draws the edges $u_0v_4$, $u_2v_2$, and $u_2v_4$, which must all be green. This yields a good $P_7$ within $10$ rounds: $v_3v_1v_2u_2v_4u_0u_1$.

	If the unit is $G_3$ and the edge $v_2v_4$ is blue, then Builder draws the edges $u_0v_4$, $u_2v_2$, and $v_2v_3$, which must all be green. This gives a good $P_7$ within $10$ rounds: $u_2v_2v_3v_1v_4u_0u_1$. If $v_2v_4$ is red, the same conclusion holds by symmetry.

	If the unit is $G_2$, $G_5$, or $G_6$, then Builder draws the edge $u_1u_2$. If $u_1u_2$ is green, then $u_0u_1u_2$ forms a green $P_3$, with $u_0$ incident to both a red and a blue edge. In this case, Builder can simply draw the edge $u_0v_3$ to obtain a good $P_7$ within $9$ rounds. If $u_1u_2$ is not green, it must be blue. 

	For $G_2$, Builder draws the edges $u_2v_3$ and $u_0v_2$, which must be green. This results in a good $P_7$ within $10$ rounds: $u_2v_3v_1v_4v_2u_0u_1$. For $G_5$ and $G_6$, Builder draws the edges $u_2v_3$ and $u_0v_4$, which must be green. This yields a good $P_7$ within $10$ rounds: $u_2v_3v_1v_2v_4u_0u_1$.
\end{proof}

\begin{theorem}\label{final3lemma} 
	For $k\ge 2$, either a bad $P_{4k-1}$ can be obtained within $6k-3$ rounds, or a good $P_{4k-1}$ can be obtained within $6k-2$ rounds.
\end{theorem}
\begin{proof}
	We proceed by induction on $k$. When $k=2$, the statement holds by Lemma~\ref{P7lemma}. For $k\ge 3$, assume the statement holds for $k-1$, i.e., either a bad $P_{4k-5}$ can be obtained within $6k-9$ rounds, or a good $P_{4k-5}$ can be obtained within $6k-8$ rounds. The desired conclusion then follows by applying Corollary~\ref{maincor}.
\end{proof}

\begin{theorem}\label{final2lemma}
	For $k\ge 1$, either a bad $P_{4k-2}$ can be obtained within $6k-4$ rounds, or a good $P_{4k-2}$ can be obtained within $6k-3$ rounds.
\end{theorem}
\begin{proof}
	We proceed by induction on $k$. When $k=1$, Builder first constructs a $P_3$. If at least one edge of this $P_3$ is green, then a green $P_2$ is obtained within $2$ rounds. If none of the edges is green, Builder extends from an internal vertex of the $P_3$ by adding an edge, which must be green. This yields a good $P_2$ within $3$ rounds.

	For $k\ge 2$, assume the statement holds for $k-1$, i.e., either a bad $P_{4k-6}$ can be obtained within $6k-10$ rounds, or a good $P_{4k-6}$ can be obtained within $6k-9$ rounds. The conclusion then follows by applying Corollary~\ref{maincor}.
\end{proof}

\begin{theorem}\label{P3P4lemma}
	$\tilde{r}(P_3,P_3,P_3)\le 4$ and $\tilde{r}(P_3,P_3,P_4)\le 6$.
\end{theorem}
\begin{proof}
	For the former, Builder constructs a star $K_{1,4}$ within $4$ rounds. To avoid a red or blue $P_3$, at least two of the edges must be green, which yields a green $P_3$.

	For the latter, Builder first creates a star $K_{1,3}$ within $3$ rounds. If it contains a green $P_3$, Builder extends one endpoint of this green $P_3$ by adding three new edges; among these, at least one must be green. This yields a green $P_4$ in at most $6$ rounds. If the $K_{1,3}$ does not contain a green $P_3$, then in order to avoid a red or blue $P_3$, the three edges of $K_{1,3}$ must consist of one red, one blue, and one green edge. Builder then adds another edge from the center of $K_{1,3}$, which must be green. Without loss of generality, denote the two green edges as $u_0u_1$ and $u_0u_2$, the red edge as $u_0u_3$, and the blue edge as $u_0u_4$. Next, Builder draws the edges $u_1u_4$ and $u_2u_4$. At least one of these must be green. By symmetry, assume $u_1u_4$ is green. Then $u_4u_1u_0u_2$ forms a green $P_4$. Hence, we conclude that $\tilde{r}(P_3,P_3,P_4)\le 6$.
\end{proof}

Combining \cref{finallemma}, \cref{final1lemma}, \cref{final3lemma}, \cref{final2lemma}, and \cref{P3P4lemma}, we obtain the upper bound in \cref{thm:main1}.\qed

\medskip

\noindent {\bf Remark}. In fact, for the cases where $\ell\not\equiv 0\pmod{4}$ and $\ell\ge 2$, or when $\ell=4$, there exists an alternative approach based on an unpublished result~\cite{Song2025}. In these cases, $N=\flo{3\ell/2}$. Builder treats red and blue as a single color, reducing the game to a two-color (non-green and green) online Ramsey game. Song, Wang, and Zhang~\cite{Song2025} proved that for $\ell\ge 2$, $\tilde{r}(K_{1,3},P_{\ell})=\flo{3\ell/2}$. Thus, Builder can force either a non-green $K_{1,3}$ or a green $P_\ell$ within $N$ rounds. In the former case, a red-blue colored $K_{1,3}$ necessarily contains a red $P_3$ or a blue $P_3$. This completes the proof.

\section{The upper bound of \cref{thm:main2}}\label{sec4}

In this section, we first present some preliminary constructions and establish several lemmas that allow us to obtain a longer green path or a green cycle by combining two vertex-disjoint green paths. We then consider four separate cases according to the residue of $\ell$ modulo $4$, and prove Theorems~\ref{c4k}, \ref{c4k+1}, \ref{c4k-2}, and \ref{c4k+3}, respectively. This completes the proof of the upper bound.

Throughout this section, we assume that no red copy of $P_3$ and no blue copy of $P_3$ appears in the graph; otherwise, Builder would have already won the game.

Before presenting the construction, we introduce three definitions that respectively describe two types of vertices, three types of green paths, and one type of cycle.

\begin{definition}
A vertex is called a \emph{good vertex} if it is incident to at least one non-green edge; it is called a \emph{better vertex} if it is incident to both a red edge and a blue edge.
\end{definition}

\begin{definition}
A green path is called a \emph{good path} if at least one of its endpoints is a good vertex; it is called a \emph{better path} if one endpoint is a good vertex and the other is a better vertex; it is called a \emph{best path} if both endpoints are better vertices.
\end{definition}

\begin{definition}
For $k\ge 2$, if a cycle $C_{k+1}$ contains a $P_3$ consisting of one red edge and one blue edge, and all the other edges in $C_{k+1}$ are green, then we denote such a cycle by $C^{\text{RB}}_{k+1}$.
\end{definition}

\begin{lemma}\label{cmainlemma} 
	Suppose there are two disjoint green paths $P_n$ and $P_k$. Then the following statements hold:
	\begin{enumerate}[label=(\arabic*).]
		\item For $n,k\ge 2$, if both $P_n$ and $P_k$ are good paths, then either a good $P_{n+k}$ is obtained in the next round, or a better $P_{n+k}$ is obtained within the next $3$ rounds.
		\item For $n,k\ge 2$, if $P_n$ is a good path, then either a good $P_{n+k}$ is obtained within the next $2$ rounds, or a better $P_{n+k}$ is obtained within the next $3$ rounds.
	    \item For $n\ge 3$ and $k\ge 2$, if $P_n$ is a better path, then a better $P_{n+k}$ is obtained in the next $2$ rounds.
	    \item For $n\ge 3$ and $k\ge 2$, Builder either obtains a green $P_{n+k}$ in the next round, a good $P_{n+k}$ within the next $2$ rounds, or a better $P_{n+k}$ within the next $4$ rounds.	
	\end{enumerate}
\end{lemma}

\begin{proof}
To prove statements (1) and (2), denote the endpoints of the green path $P_n$ by $u_1$ and $u_2$, and those of $P_k$ by $v_1$ and $v_2$. When $P_n$ is a good path, assume that $u_1$ is its good vertex; when $P_k$ is a good path, assume that $v_1$ is its good vertex. For the proofs of statements (3) and (4), we label all vertices of the two paths and write the green paths as $P_n=u_1u_2\cdots u_n$ and $P_k=v_1v_2\cdots v_k$. We now prove each of the four statements in turn.

\noindent {\bf Proof of (1).} Builder first constructs the edge $v_1u_2$. If this edge is colored green, then a good $P_{n+k}$ is obtained in $1$ round: $u_1P_nu_2v_1P_{k}v_2$. If the edge is colored non-green, then $u_2$ becomes a good vertex and $v_1$ becomes a better vertex. In this case, Builder proceeds to construct the edge $u_1v_1$, which must be green. Then, Builder constructs the edge $u_2v_2$: if this edge is non-green, the path $v_2P_kv_1u_1P_n u_2$ forms a better $P_{n+k}$; if $u_2v_2$ is green, the path $u_1P_n u_2v_2P_kv_1$ also forms a better $P_{n+k}$. The process takes at most $3$ rounds.

\noindent {\bf Proof of (2).} Builder first constructs the edges $u_2v_1$ and $u_2v_2$. If at least one of these edges is colored green, say $u_2v_1$, then a good $P_{n+k}$ is obtained in $2$ rounds: $u_1P_nu_2v_1P_{k}v_2$. If both $u_2v_1$ and $u_2v_2$ are colored non-green, then $u_2$ becomes a better vertex since it is incident to both a red and a blue edge. In this situation, since $u_1$ must connect to at least one of $v_1$ or $v_2$ via a green edge (say $v_1$), Builder constructs the edge $u_1v_1$, which must be green. This results in a better $P_{n+k}$ in $3$ rounds: $u_2P_nu_1v_1P_{k}v_2$.

\noindent {\bf Proof of (3).} Assume $P_n$ is a better path with $u_1$ as a good vertex and $u_n$ as a better vertex. Builder first constructs the edge $u_1v_1$. If it is non-green, then $u_1$ becomes a better vertex. Builder then constructs the edge $u_1v_k$, which must be green. This yields a better $P_{n+k}$ in $2$ rounds: $u_nP_nu_1v_kP_{k}v_1$. If $u_1v_1$ is green, Builder constructs the edge $v_ku_2$: if this edge is green, then the path $u_n\cdots u_2v_kP_kv_1u_1$ is a better $P_{n+k}$; if $v_ku_2$ is non-green, then the path $u_nP_nu_1v_1P_kv_k$ forms a better $P_{n+k}$. The process takes at most $2$ rounds.

\noindent {\bf Proof of (4).} Builder first constructs the edge $v_1u_1$. If it is green, then a green $P_{n+k}$ is obtained in $1$ round: $u_nP_nu_1v_1P_kv_k$. If $v_1u_1$ is non-green, Builder proceeds with the edge $u_1v_k$. If this edge is green, then a good $P_{n+k}$ is obtained in $2$ rounds: $u_nP_nu_1v_kP_kv_1$. If the edge $u_1v_k$ is non-green, then $v_1$ and $v_k$ are both good vertices, and $u_1$ becomes a better vertex. Builder then constructs the edge $v_ku_2$: if it is green, Builder then constructs the necessarily green edge $u_1u_n$, and a better $P_{n+k}$ is obtained in $4$ rounds: $v_1P_kv_ku_2P_{n-1}u_nu_1$; if $v_ku_2$ is non-green, Builder constructs the necessarily green edge $u_n v_k$, which also leads to a better $P_{n+k}$ in $4$ rounds: $u_1P_nu_nv_kP_kv_1$.
\end{proof}

\begin{lemma}\label{Tlemma} 
Suppose there exists a $C^{\text{RB}}_{k+1}$ and a disjoint green path $P_n$, where $k,n\ge 2$. Then Builder can obtain a better $P_{n+k+1}$ within $3$ rounds. In particular, if $P_n$ is a good path, then Builder can obtain a better $P_{n+k+1}$ within $2$ rounds.
\end{lemma}

\begin{proof}
Let the cycle $v_1\cdots v_kv_{k+1}v_1$ be a $C^{\text{RB}}_{k+1}$. Without loss of generality, assume that the edges $v_1v_{k+1}$ and $v_{k+1}v_k$ are red and blue, respectively, so that $v_1\cdots v_k$ forms a green path $P_k$. Let the endpoints of $P_n$ be $u_1$ and $u_2$, and if $P_n$ is a good path, assume without loss of generality that $u_1$ is its good vertex.

Builder draws the edges $v_1u_1$ and $v_1u_2$. Since $v_1$ is a good vertex, at least one of these edges must be green. Without loss of generality, assume $v_1u_1$ is green. Then Builder draws the edge $u_2v_{k+1}$, which must also be green. Hence, within $3$ rounds, we obtain a better $P_{n+k+1}$: $v_{k+1}u_2P_nu_1v_1P_kv_k$.

If $P_n$ is a good path, it is easy to see (without drawing) that either $u_1v_1$ or $u_1v_k$ has to be green. Without loss of generality, assume that $u_1v_1$ is green. Then Builder draws the edges $u_1v_1$ and $u_2v_{k+1}$, both of which must be green. Therefore, within $2$ rounds, Builder obtains a better $P_{n+k+1}$: $v_{k+1}u_2P_nu_1v_1P_kv_k$.
\end{proof}

\begin{lemma}\label{cbasiclemma} 
	For $i\in [6]$, the graphs $G_i$ are illustrated in Figure~\ref{fig:c-unit}. Builder can construct $G_1$ or $G_2$ within $3$ rounds; otherwise, after one additional round, one of $G_3$, $G_4$, $G_5$, or $G_6$ will inevitably be constructed.
\end{lemma}

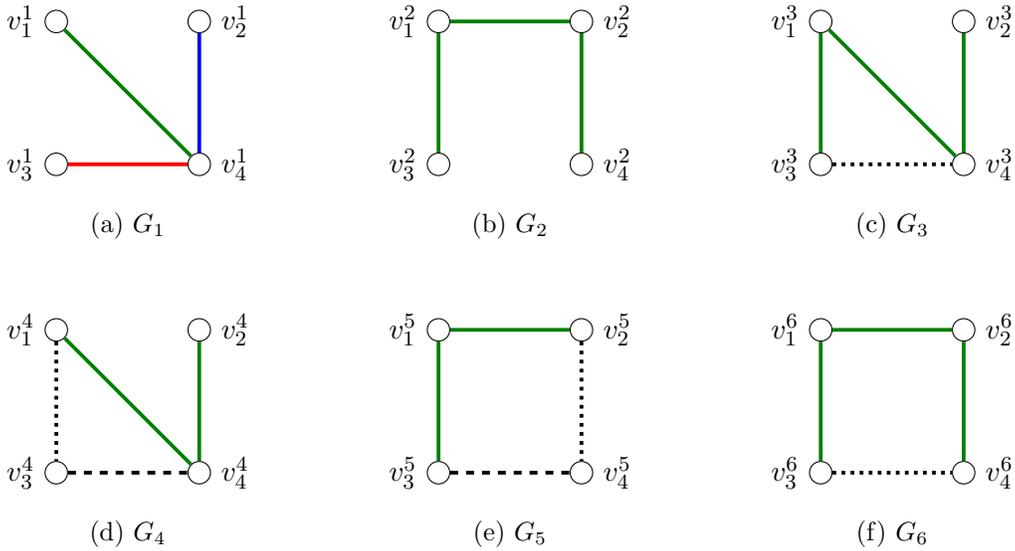
\begin{figure}[ht]
	\centering
	\definecolor{darkgreen}{rgb}{0.0, 0.5, 0.0} 
	\begin{subfigure}{0.3\textwidth}
		\centering
		\begin{tikzpicture}[scale=1.9, 
			vertex/.style={circle, draw, fill=white, inner sep=3pt},
			every label/.style={scale=1, black}]
			
			\node[vertex, label=left:$v_1^1$] (v1) at (0, 1) {};
			\node[vertex, label=right:$v_2^1$] (v2) at (1, 1) {};
			\node[vertex, label=left:$v_3^1$] (v3) at (0, 0) {};
			\node[vertex, label=right:$v_4^1$] (v4) at (1, 0) {};
			
			\draw[line width=0.5mm, darkgreen] (v1) -- (v4);
			\draw[line width=0.5mm, blue] (v2) -- (v4);
			\draw[line width=0.5mm, red] (v3) -- (v4);
		\end{tikzpicture}
		\caption{$G_1$}
	\end{subfigure}
	\begin{subfigure}{0.3\textwidth}
		\centering
		\begin{tikzpicture}[scale=1.9, 
			vertex/.style={circle, draw, fill=white, inner sep=3pt},
			every label/.style={scale=1, black}]
			
			\node[vertex, label=left:$v_1^2$] (v1) at (0, 1) {};
			\node[vertex, label=right:$v_2^2$] (v2) at (1, 1) {};
			\node[vertex, label=left:$v_3^2$] (v3) at (0, 0) {};
			\node[vertex, label=right:$v_4^2$] (v4) at (1, 0) {};
			
			\draw[line width=0.5mm, darkgreen] (v1) -- (v2);
			\draw[line width=0.5mm, darkgreen] (v1) -- (v3);
			\draw[line width=0.5mm, darkgreen] (v2) -- (v4);
	
		\end{tikzpicture}
		\caption{$G_2$}
	\end{subfigure}
	\begin{subfigure}{0.3\textwidth}
		\centering
		\begin{tikzpicture}[scale=1.9, 
			vertex/.style={circle, draw, fill=white, inner sep=3pt},
			every label/.style={scale=1, black}]
			
			\node[vertex, label=left:$v_1^3$] (v1) at (0, 1) {};
			\node[vertex, label=right:$v_2^3$] (v2) at (1, 1) {};
			\node[vertex, label=left:$v_3^3$] (v3) at (0, 0) {};
			\node[vertex, label=right:$v_4^3$] (v4) at (1, 0) {};
			
			\draw[line width=0.5mm, darkgreen] (v1) -- (v3);
			\draw[line width=0.5mm, darkgreen] (v2) -- (v4);
			\draw[line width=0.5mm, dotted] (v3) -- (v4);
			\draw[line width=0.5mm, darkgreen] (v1) -- (v4);
		\end{tikzpicture}
		\caption{$G_3$}
	\end{subfigure}

	\vspace{2em}
		\begin{subfigure}{0.3\textwidth}
		\centering
		\begin{tikzpicture}[scale=1.9, 
			vertex/.style={circle, draw, fill=white, inner sep=3pt},
			every label/.style={scale=1, black}]
			
			\node[vertex, label=left:$v_1^4$] (v1) at (0, 1) {};
			\node[vertex, label=right:$v_2^4$] (v2) at (1, 1) {};
			\node[vertex, label=left:$v_3^4$] (v3) at (0, 0) {};
			\node[vertex, label=right:$v_4^4$] (v4) at (1, 0) {};
			
			\draw[line width=0.5mm, dotted] (v1) -- (v3);
			\draw[line width=0.5mm, darkgreen] (v2) -- (v4);
			\draw[line width=0.5mm, dashed] (v3) -- (v4);
			\draw[line width=0.5mm, darkgreen] (v1) -- (v4);
		\end{tikzpicture}
		\caption{$G_4$}
	\end{subfigure}
	\begin{subfigure}{0.3\textwidth}
		\centering
		\begin{tikzpicture}[scale=1.9, 
			vertex/.style={circle, draw, fill=white, inner sep=3pt},
			every label/.style={scale=1, black}]
			
			\node[vertex, label=left:$v_1^5$] (v1) at (0, 1) {};
			\node[vertex, label=right:$v_2^5$] (v2) at (1, 1) {};
			\node[vertex, label=left:$v_3^5$] (v3) at (0, 0) {};
			\node[vertex, label=right:$v_4^5$] (v4) at (1, 0) {};
			
			\draw[line width=0.5mm, darkgreen] (v1) -- (v2);
			\draw[line width=0.5mm, darkgreen] (v1) -- (v3);
			\draw[line width=0.5mm, dotted] (v2) -- (v4);
			\draw[line width=0.5mm, dashed] (v3) -- (v4);
		\end{tikzpicture}
		\caption{$G_5$}
	\end{subfigure}
	\begin{subfigure}{0.3\textwidth}
		\centering
		\begin{tikzpicture}[scale=1.9, 
			vertex/.style={circle, draw, fill=white, inner sep=3pt},
			every label/.style={scale=1, black}]
			
			\node[vertex, label=left:$v_1^6$] (v1) at (0, 1) {};
			\node[vertex, label=right:$v_2^6$] (v2) at (1, 1) {};
			\node[vertex, label=left:$v_3^6$] (v3) at (0, 0) {};
			\node[vertex, label=right:$v_4^6$] (v4) at (1, 0) {};
			
			\draw[line width=0.5mm, darkgreen] (v1) -- (v2);
			\draw[line width=0.5mm, darkgreen] (v1) -- (v3);
			\draw[line width=0.5mm, darkgreen] (v2) -- (v4);
			\draw[line width=0.5mm, dotted] (v3) -- (v4);
		\end{tikzpicture}
		\caption{$G_6$}
	\end{subfigure}
	\caption{Graphs Builder can create in 4 rounds, where $\{\text{dotted}, \text{dashed}\} = \{\text{red}, \text{blue}\}$.}
	\label{fig:c-unit}
\end{figure}

\begin{proof}	
	Note that in graphs $G_3$ and $G_6$, the dashed edge represents an edge that may be either red or blue. In graphs $G_4$ and $G_5$, the two dashed edges represent a red edge and a blue edge, respectively. Also observe that the vertex labels in the graphs are provided solely for convenience in the subsequent proof and are not related to the order in which the edges are drawn.
	
	Builder first constructs a $P_3$, and the edges of this $P_3$ fall into one of the following three cases: one red and one blue edge, exactly one green edge, or both edges green.
	\begin{enumerate}[label=(\arabic*).]
		\item If the two edges of the $P_3$ are one red and one blue, Builder connects the internal vertex of this $P_3$ to a new vertex. This new edge must be green, resulting in the graph $G_1$.
		\item If the $P_3$ contains exactly one green edge, Builder connects the internal vertex of this $P_3$ to a new vertex, forming an edge $e_1$. If $e_1$ is not green, then $G_1$ is obtained. If $e_1$ is green, Builder then adds an edge between the two endpoints of the original $P_3$. If this edge is green, the resulting graph is $G_3$; if it is not green, the resulting graph is $G_4$.
		\item If both edges of the $P_3$ are green, let the endpoints of $P_3$ be $x$ and $y$. Builder then introduces a new vertex $w$ and draws the edge $wy$. If $wy$ is green, this results in graph $G_2$. If $wy$ is non-green, Builder then draws the edge $wx$. If $wx$ is non-green, graph $G_5$ is obtained; if $wx$ is green, graph $G_6$ is obtained.\qedhere
	\end{enumerate}
\end{proof}

\begin{lemma}\label{mainthm}
Builder can construct either $G_1\cup G_1$ within 6 rounds (where $G_1$ is shown in Figure~\ref{fig:c-unit} (a)), or a green $P_8$ within 7 rounds, or a good $P_8$ within 9 rounds, or a better $P_8$ within 11 rounds.
\end{lemma}

\begin{proof} 
Builder applies Lemma~\ref{cbasiclemma} twice to construct two vertex-disjoint copies of graphs from Figure~\ref{fig:c-unit}, denoted as $(G_i,G_j)$. Without loss of generality, we may assume $1\leq i\leq j \leq 6$ by symmetry. We consider all possible configurations of $(G_i,G_j)$ as follows.

\setcounter{case}{0}
\begin{case}
$i=6$.
\end{case}
In this case, we only need to consider $(G_6,G_6)$. By applying Lemma~\ref{cmainlemma} (1), Builder either obtains a good $P_8$ within 9 rounds or a better $P_8$ within 11 rounds.

\begin{case}
$i=5$.
\end{case}
For the pair $(G_5,G_6)$, observe that $G_5$ is a $C^{\text{RB}}_4$, so Lemma~\ref{Tlemma} applies, and Builder obtains a better $P_8$ within 10 rounds. For $(G_5,G_5)$, we relabel the vertices in one copy of $G_5$ by replacing $v_i^5$ with $u_i^5$ for $i\in [4]$. It is easy to see (without drawing) that $v_2^5$ can be connected by a green edge to either $u_2^5$ or $u_3^5$; without loss of generality, assume it is $u_2^5$. Then Builder draws edges $v_2^5u_2^5$, $v_3^5u_4^5$, and $v_4^5u_3^5$, all of which must be green. This yields a best $P_8$ in 11 rounds: $v_4^5u_3^5u_1^5u_2^5v_2^5v_1^5v_3^5u_4^5$. Clearly, a best $P_8$ is also a better $P_8$.

\begin{case}
$i=4$.
\end{case}
For both $(G_4,G_5)$ and $(G_4,G_6)$, Builder first draws the edge $v_2^4v_3^4$, which must be green. Then, Lemma~\ref{Tlemma} applies to $(G_4,G_5)$, and Lemma~\ref{cmainlemma} (3) applies to $(G_4,G_6)$, both yielding a better $P_8$ within 11 rounds.

For the pair $(G_4,G_4)$, we relabel the vertices of one copy of $G_4$ by replacing $v_i^4$ with $u_i^4$ for $i\in [4]$. Builder then draws the edges $v_2^4v_3^4$, $u_2^4u_3^4$, and $v_1^4u_3^4$, all of which must be green. This results in a better $P_8$ within 11 rounds: $u_1^4u_4^4u_2^4u_3^4v_1^4v_4^4v_2^4v_3^4$.

\begin{case}
$i=3$.
\end{case}
For $(G_3,G_3)$, $(G_3,G_5)$, and $(G_3,G_6)$, we apply Lemma~\ref{cmainlemma} (1), Lemma~\ref{Tlemma}, and Lemma~\ref{cmainlemma} (1), respectively. For the case $(G_3,G_4)$, Builder draws the edge $v_3^4v_2^4$, which must be green, and then applies Lemma~\ref{cmainlemma} (3). In all these cases, either a good $P_8$ is constructed within 9 rounds or a better $P_8$ is constructed within 11 rounds.

\begin{case}
$i=2$.
\end{case}
For $(G_2,G_2)$, $(G_2,G_3)$, $(G_2,G_5)$, and $(G_2,G_6)$, we apply Lemma~\ref{cmainlemma} (4), Lemma~\ref{cmainlemma} (2), Lemma~\ref{Tlemma}, and Lemma~\ref{cmainlemma} (2), respectively, which proves the claim. For the case $(G_2,G_4)$, Builder draws the edges $v_2^4v_3^4$ and $v_3^4v_3^2$, both of which must be green. This yields a good $P_8$ in 9 rounds: $v_4^2v_2^2v_1^2v_3^2v_3^4v_2^4v_4^4v_1^4$.

\begin{case}
    $i=1$.
\end{case}
We first establish the following claim.
\begin{claim}\label{G1}
    Given $G_1$ and a vertex-disjoint green path $P_n$ with $n\ge 2$, Builder can obtain a better $P_{n+4}$ within 5 rounds. In particular, if $P_n$ is a good path, then a better $P_{n+4}$ can be obtained within 4 rounds.
\end{claim}

\begin{proof}
    Let $u$ and $v$ be the two endpoints of the green path $P_n$. If $P_n$ is a good path, we may assume without loss of generality that $u$ is a good vertex.
    
    Builder first draws the edge $v_1^1v_3^1$. If this edge is green, Builder then draws $v_3^1v_2^1$, which must be green as well. Thus, the path $v_4^1v_1^1v_3^1v_2^1$ forms a better $P_4$. Applying Lemma~\ref{cmainlemma} (3), this path can be combined with $P_n$ in two additional rounds to form a better $P_{n+4}$. The process takes 4 rounds.

    If $v_1^1v_3^1$ is not green, Builder draws the edges $v_2^1u$ and $v_2^1v$. At least one of them must be green; without loss of generality, suppose $v_2^1u$ is green. Then Builder draws $v_2^1v_3^1$ and $v_4^1v$, both of which must be green. This results in a better $P_{n+4}$ within 5 rounds: $v_3^1v_2^1uP_n vv_4^1v_1^1$.

    In the special case where the green $P_n$ is a good path, without loss of generality, assume that the good vertex $u$ is incident to a red edge. Note that $v_3^1$ already becomes a better vertex. Builder then draws the edges $v_2^1v_3^1$, $v_2^1u$, and $v_4^1v$, all of which must be green. This yields a better $P_{n+4}$ in 4 rounds: $v_3^1v_2^1uP_n vv_4^1v_1^1$.
\end{proof}

For the pair $(G_1,G_1)$, a disjoint union of two copies of $G_1$ can be obtained within 6 rounds.

For the pairs $(G_1,G_2)$, $(G_1,G_3)$, and $(G_1,G_6)$, Claim~\ref{G1} applies, and in each case a better $P_8$ can be obtained within 11 rounds.

For the pair $(G_1,G_4)$, it is easy to see (without drawing) that one of the edges $v_1^4v_3^1$ and $v_1^4v_2^1$ must be green. Without loss of generality, assume $v_1^4v_3^1$ is green. Then Builder draws the edges $v_2^1v_3^1$, $v_2^4v_3^4$, $v_3^1v_1^4$, and $v_3^4v_1^1$, all of which must be green. Consequently, a better $P_8$ is obtained within 11 rounds: $v_4^1v_1^1v_3^4v_2^4v_4^4v_1^4v_3^1v_2^1$.

Finally, for the pair $(G_1,G_5)$, it is easy to see (without drawing) that one of the edges $v_3^1v_2^5$ and $v_3^1v_3^5$ must be green. Without loss of generality, assume $v_3^1v_3^5$ is green. Builder then draws the edges $v_3^5v_3^1$, $v_3^1v_2^1$, $v_2^1v_4^5$, and $v_4^5v_1^1$, all of which must be green. Thus, a better $P_8$ is constructed within 11 rounds: $v_4^1v_1^1v_4^5v_2^1v_3^1v_3^5v_1^5v_2^5$.
\end{proof}

\begin{lemma}\label{G1andG1}
Suppose we are given $G_1\cup G_1$ (where $G_1$ is shown in Figure~\ref{fig:c-unit} (a)). Then within the next 6 rounds, Builder can construct a better $P_8$. Moreover, Builder can determine the color of the non-green edge incident to the good endpoint of this $P_8$.
\end{lemma}	
\begin{proof}
For convenience, we relabel the vertices in one copy of $G_1$ by replacing $v_i^1$ with $u_i^1$ for $i\in [4]$.

We now describe a strategy for Builder to obtain a better $P_8$ with the good endpoint incident to a blue edge within the next 6 rounds.

Builder first draws the edge $v_3^1v_1^1$. If $v_3^1v_1^1$ is green, then Builder proceeds by drawing the edges $v_3^1v_2^1$, $v_2^1u_3^1$, $u_3^1u_2^1$, and $v_4^1u_1^1$, all of which must be green. Thus, within the next 5 rounds, a better $P_8$ with the good endpoint incident to a blue edge is formed: $u_4^1u_1^1v_4^1v_1^1v_3^1v_2^1u_3^1u_2^1$.

If $v_3^1v_1^1$ is not green, then it must be blue. Builder then draws the edges $v_3^1u_2^1$, $u_2^1u_3^1$, $u_3^1v_2^1$, $v_2^1u_4^1$, and $v_4^1u_1^1$, all of which must be green. In this case, within the next 6 rounds, a better $P_8$ with the good endpoint incident to a blue edge is constructed: $v_3^1u_2^1u_3^1v_2^1u_4^1u_1^1v_4^1v_1^1$.

By the symmetry in the coloring of the edges in $G_1$, it is easy to see that Builder can similarly construct a better $P_8$ with the good endpoint incident to a red edge within the next 6 rounds.
\end{proof}

\begin{lemma}\label{low Connection Lemma}
Suppose there is a green path $P_k$: $v_1v_2\cdots v_k$ with $k\ge 6$. Then a green cycle $C_k$ can be constructed in at most 5 rounds. In particular, if $P_k$ is a good path and there is no edge between its endpoints, then a green $C_k$ can be constructed in at most 4 rounds.
\end{lemma}
\begin{proof}
Builder first draws the edge $v_1v_k$. If $v_1v_k$ is green, then a green $C_k$ is obtained in 1 round. If $v_1v_k$ is not green, Builder proceeds by drawing the edge $v_1v_4$. If $v_1v_4$ is also not green, then the edges $v_1v_5$ and $v_4v_k$, which must be green, are drawn. Thus, a green $C_k$ is formed in 4 rounds: $v_1v_5\cdots v_kv_4v_3v_2v_1$.

If $v_1v_4$ is green, Builder continues by drawing the edge $v_3v_k$. If $v_3v_k$ is also green, then a green $C_k$ is formed in 3 rounds: $v_kv_3v_2v_1v_4\cdots v_k$. If $v_3v_k$ is not green, then Builder draws the edges $v_1v_3$ and $v_2v_k$, which must be green. Hence, a green $C_k$ is constructed in 5 rounds: $v_kv_2v_1v_3\cdots v_k$.

Now consider the case where $P_k$ is a good path and there is no edge between its endpoints. Without loss of generality, assume that $v_1$ is a good vertex. Builder first draws the edge $v_1v_k$. If $v_1v_k$ is green, a green $C_k$ is obtained in 1 round.

If $v_1v_k$ is not green, then $v_1$ becomes a better vertex. Builder draws the edge $v_3v_k$. If $v_3v_k$ is green, Builder then draws the edge $v_1v_4$, which must be green. Hence, a green $C_k$ is obtained in 3 rounds: $v_3v_2v_1v_4\cdots v_kv_3$. If $v_3v_k$ is not green, Builder draws the edges $v_1v_3$ and $v_2v_k$, both of which must be green. Thus, a green $C_k$ is formed in 4 rounds: $v_2v_1v_3v_4\cdots v_kv_2$.
\end{proof}

\begin{lemma}\label{Connection Lemma}
Let $P_n$ and $P_k$ be two vertex-disjoint green paths with $n,k\ge 3$. Then the following statements hold:
\begin{enumerate}[label=(\arabic*).]
    \item If both $P_n$ and $P_k$ are good paths, then a green $C_{n+k}$ can be constructed within the next $5$ rounds.
    \item If $P_n$ is a good path, then a green $C_{n+k}$ can be constructed within the next $6$ rounds.
    \item A green $C_{n+k}$ can be constructed from $P_n$ and $P_k$ within the next $7$ rounds.
    \item If $P_n$ is a better path, then a green $C_{n+k}$ can be constructed within the next $3$ rounds.
\end{enumerate}
\end{lemma}
\begin{proof}
Let $P_n$ be denoted by $u_1u_2\cdots u_n$ and $P_k$ by $v_1v_2\cdots v_k$. If $P_n$ (respectively $P_k$) is a good path, assume without loss of generality that $u_n$ (respectively $v_k$) is a good vertex.

\noindent {\bf Proof of (1).} Builder draws the edge $v_1u_n$. If $v_1u_n$ is green, then Lemma~\ref{low Connection Lemma} applies, and a green $C_{n+k}$ is obtained within $5$ rounds. If $v_1u_n$ is not green, then $u_n$ becomes a better vertex. Builder then draws the edge $v_1u_1$. If $v_1u_1$ is green, Builder continues by drawing the edge $u_n v_k$, which must be green. Thus, a green $C_{n+k}$ is obtained in $3$ rounds: $v_1u_1P_n u_n v_kP_kv_1$. If $v_1u_1$ is not green, Builder draws the edges $u_1u_n$, $v_1u_{n-1}$, and $v_ku_n$, all of which must be green. Hence, a green $C_{n+k}$ is formed in $5$ rounds: $v_1u_{n-1}P_{n-1} u_1u_n v_kP_kv_1$.

\noindent {\bf Proof of (2).} Builder draws the edge $v_1u_n$. If $v_1u_n$ is green, then Lemma~\ref{low Connection Lemma} applies, and a green $C_{n+k}$ is obtained within $6$ rounds. If $v_1u_n$ is not green, then $u_n$ becomes a better vertex. Builder draws the edge $v_ku_n$, which must be green. Lemma~\ref{low Connection Lemma} again applies, so a green $C_{n+k}$ is obtained within $6$ rounds.

\noindent {\bf Proof of (3).} Builder draws the edge $v_1u_n$. If $v_1u_n$ is green, then by Lemma~\ref{low Connection Lemma}, a green $C_{n+k}$ can be obtained within $6$ rounds.  
If $v_1u_n$ is not green, Builder proceeds by drawing $v_1u_1$. If $v_1u_1$ is green, then Lemma~\ref{low Connection Lemma} still applies, yielding a green $C_{n+k}$ within $6$ rounds.  
If $v_1u_1$ is not green, Builder draws the edge $u_n v_k$. If $u_n v_k$ is green, then the edges $u_1u_n$ and $v_1u_{n-1}$, both necessarily green, are drawn, resulting in a green $C_{n+k}$ in $5$ rounds:  
$u_1\cdots u_{n-1}v_1\cdots v_ku_n u_1$.  
If $u_n v_k$ is not green, Builder draws the edges $v_1v_k$, $u_1u_n$, $u_1v_{k-1}$, and $v_ku_2$, all of which must be green. Thus, a green $C_{n+k}$ is formed in $7$ rounds:  
$u_2\cdots u_nu_1v_{k-1}\cdots v_1v_ku_2$.

\noindent {\bf Proof of (4).} Assume $u_1$ is the better vertex. Builder draws the edges $u_n v_k$ and $u_n v_1$. At least one of these must be green; without loss of generality, suppose $v_ku_n$ is green. Builder then draws the edge $v_1u_n$, which must also be green. Clearly, a green $C_{n+k}$ is obtained in $3$ rounds: $u_1\cdots u_nv_k\cdots v_1u_1$.
\end{proof}

\begin{theorem}\label{c4k}
For all $k\ge 4$, we have $\tilde{r}(P_3,P_3,C_{4k}) \leq 6k+1$.
\end{theorem}

\begin{proof}
By Theorem~\ref{finallemma}, Builder can either construct a green $P_{4k-8}$ within $6k-14$ rounds, or a good $P_{4k-8}$ within $6k-13$ rounds. Then, by applying Lemmas~\ref{mainthm} and~\ref{G1andG1}, Builder can construct a green $P_8$ that is vertex-disjoint from $P_{4k-8}$.

If Builder obtains a green $P_{4k-8}$ within $6k-14$ rounds, and the $P_8$ is a green path constructed within 7 rounds, a good path constructed within 9 rounds, or a better path constructed within 11 rounds, then by applying parts (3), (2), and (4) of Lemma~\ref{Connection Lemma} respectively, Builder can connect $P_8$ and $P_{4k-8}$ into a green $C_{4k}$. In all cases, the total number of rounds does not exceed $6k+1$.

If Builder obtains a good $P_{4k-8}$ within $6k-13$ rounds, and the $P_8$ is a green path constructed within 7 rounds, a good path within 9 rounds, or a better path within 11 rounds, then by applying parts (2), (1), and (4) of Lemma~\ref{Connection Lemma} respectively, a green $C_{4k}$ can be formed. Again, the total number of rounds is at most $6k+1$.

If Builder obtains a green $P_{4k-8}$ within $6k-14$ rounds, and the $P_8$ is a better path constructed within 12 rounds, then applying Lemma~\ref{Connection Lemma} (4) completes the proof, with a total of at most $6k+1$ rounds.

If Builder obtains a good $P_{4k-8}$ within $6k-13$ rounds, and the $P_8$ is a better path constructed within 12 rounds, then only two additional rounds are needed to form the green $C_{4k}$. This is because if the good endpoint of $P_{4k-8}$ is incident to a red edge, then by Lemma~\ref{G1andG1}, the good endpoint of $P_8$ can be made incident to a blue edge; conversely, if the good endpoint of $P_{4k-8}$ is incident to a blue edge, then the good endpoint of $P_8$ can be made incident to a red edge. Therefore, the total number of rounds is at most $6k+1$.
\end{proof}

\begin{lemma}\label{c4k1cor}
	Builder can either construct a green $P_9$ within 10 rounds, a better $P_9$ within 13 rounds, or a best $P_9$ within 14 rounds.
\end{lemma}
\begin{proof}
	Before proceeding with the proof, we first establish the following claim.
	\begin{claim}\label{better}
		For any $n\ge 3$, every better $P_n$ can be extended to a better $P_{n+1}$ within 2 rounds.
	\end{claim}
	\begin{proof}
		Denote the better $P_n$ as $v_1v_2\cdots v_n$. Suppose vertex $v_1$ is a good vertex and $v_n$ is a better vertex.

		Introduce a new vertex $w$ and draw the edge $wv_1$. If $wv_1$ is not green, then $v_1$ becomes a better vertex. In this case, draw the edge $wv_n$, which must be green, yielding a better $P_{n+1}$: $wv_n\cdots v_2v_1$ within 2 rounds. If $wv_1$ is green, then draw the edge $wv_2$. If $wv_2$ is not green, the resulting better $P_{n+1}$ is $wv_1v_2\cdots v_n$; if $wv_2$ is green, then we obtain $v_1wv_2\cdots v_n$ as a better $P_{n+1}$.
	\end{proof}

	First apply Lemma~\ref{mainthm}, and let the resulting graph be $G$.

	If $G$ is $G_1\cup G_1$, relabel the vertices of one copy of $G_1$ from $v_i^1$ to $u_i^1$ for $i\in [4]$. Builder draws the edge $v_3^1v_1^1$. If $v_3^1v_1^1$ is green, then draw the edges $v_3^1v_2^1$, $v_2^1u_3^1$, $u_3^1u_2^1$, and $v_4^1u_1^1$, all of which must be green. This yields a better $P_8$ within 11 rounds: $u_4^1u_1^1v_4^1v_1^1v_3^1v_2^1u_3^1u_2^1$. Applying Claim~\ref{better}, a better $P_9$ is then obtained within 13 rounds. 

	If $v_3^1v_1^1$ is not green, introduce a new vertex $w$ and draw the edge $v_1^1v_2^1$. If $v_1^1v_2^1$ is green, then draw the necessarily green edges $v_4^1w$, $wv_3^1$, $v_3^1u_1^1$, $u_3^1u_2^1$, and $v_2^1u_3^1$, resulting in a better $P_9$ within 13 rounds: $u_4^1u_1^1v_3^1wv_4^1v_1^1v_2^1u_3^1u_2^1$. If $v_1^1v_2^1$ is not green, then draw the necessarily green edges $u_1^1v_2^1$, $v_2^1u_3^1$, $u_3^1u_2^1$, $u_2^1v_3^1$, $v_3^1w$, and $wv_4^1$, yielding a best $P_9$ within 14 rounds: $u_4^1u_1^1v_2^1u_3^1u_2^1v_3^1wv_4^1v_1^1$.

	If $G$ is a green $P_8$, denote it by $v_1v_2\cdots v_8$. Introduce a new vertex $v_9$ and draw the edges $v_1v_9$ and $v_8v_9$. If at least one of these edges is green, a green $P_9$ can be constructed within 9 rounds. If both are not green, then draw the necessarily green edges $v_1v_8$ and $v_7v_9$, resulting in a better $P_9$ within 11 rounds: $v_8v_1\cdots v_7v_9$.

	If $G$ is a good $P_8$, denote it by $v_1v_2 \cdots v_8$, where $v_8$ is a good vertex. Introduce a new vertex $v_9$ and draw the edge $v_8v_9$. If $v_8v_9$ is green, then a green $P_9$ is constructed within 10 rounds. If $v_8v_9$ is not green, draw the edge $v_1v_9$. If $v_1v_9$ is green, then a better $P_9$ is obtained within 11 rounds: $v_9v_1\cdots v_8$. If $v_1v_9$ is not green, then draw the necessarily green edges $v_1v_8$ and $v_7v_9$, resulting in a better $P_9$ within 13 rounds: $v_8v_1\cdots v_7v_9$.

	If $G$ is a better $P_8$, then by directly applying Claim~\ref{better}, a better $P_9$ can be constructed within 13 rounds.
\end{proof}

\begin{theorem}\label{c4k+1}
For all $k\ge 4$, we have $\tilde{r}(P_3,P_3,C_{4k+1}) \leq 6k+3$.
\end{theorem}

\begin{proof}
By Theorem~\ref{finallemma}, Builder can either construct a green $P_{4k-8}$ within $6k-14$ rounds, or a good $P_{4k-8}$ within $6k-13$ rounds. Then, using Lemma~\ref{c4k1cor}, Builder constructs a $P_9$ that is vertex-disjoint from $P_{4k-8}$.

If Builder obtains a green $P_{4k-8}$ within $6k-14$ rounds and the $P_9$ is a green path constructed within 10 rounds or a better path constructed within 13 rounds, then by applying parts (3) and (4) of Lemma~\ref{Connection Lemma} respectively, Builder can connect $P_{4k-8}$ and $P_9$ into a green $C_{4k+1}$ within a total of $6k+3$ rounds.

If Builder obtains a good $P_{4k-8}$ within $6k-13$ rounds and the $P_9$ is a green path constructed within 10 rounds or a better path constructed within 13 rounds, then by applying parts (2) and (4) of Lemma~\ref{Connection Lemma} respectively, a green $C_{4k+1}$ can be formed within $6k+3$ rounds in total.

If the $P_9$ is a best path constructed within 14 rounds, it is straightforward to complete the connection to $P_{4k-8}$ in 2 additional rounds, yielding a green $C_{4k+1}$ within a total of $6k+3$ rounds.
\end{proof}

\begin{lemma}\label{cbasiclemma2}
	Builder can construct $H_1$ within $2$ rounds; otherwise, by proceeding one additional round, Builder can construct either $H_2$ or $H_3$.
\end{lemma}

	\begin{figure}[ht]
		\centering
		\definecolor{darkgreen}{rgb}{0.0, 0.5, 0.0} 
		\begin{subfigure}{0.3\textwidth}
			\centering
			\begin{tikzpicture}[scale=1.9, 
				vertex/.style={circle, draw, fill=white, inner sep=3pt},
				every label/.style={scale=1, black}]
				\node[vertex, label=below:$v_1^1$] (v1) at (0, 0) {};
				\node[vertex, label=below:$v_2^1$] (v2) at (1, 0) {};
				\node[vertex, label=below:$v_3^1$] (v3) at (2, 0) {};
				\draw[line width=0.5mm, darkgreen] (v1) -- (v2);
				\draw[line width=0.5mm,darkgreen] (v2) -- (v3);
			\end{tikzpicture}
			\caption{$H_1$}
		\end{subfigure}
		\begin{subfigure}{0.3\textwidth}
			\centering
			\begin{tikzpicture}[scale=1.9, 
				vertex/.style={circle, draw, fill=white, inner sep=3pt},
				every label/.style={scale=1, black}]		
					\node[vertex, label=below:$v_1^2$] (v1) at (0, 0) {};
				\node[vertex, label=below:$v_2^2$] (v2) at (2, 0) {};
				\node[vertex, label=left:$v_3^2$] (v3) at (1, 1) {};
			
				\draw[line width=0.5mm, darkgreen] (v1) -- (v2);
				\draw[line width=0.5mm,dashed] (v1) -- (v3);
				\draw[line width=0.5mm,dotted] (v2) -- (v3);
			
			\end{tikzpicture}
			\caption{$H_2$}
		\end{subfigure}
		\begin{subfigure}{0.3\textwidth}
		\centering
		\begin{tikzpicture}[scale=1.9, 
			vertex/.style={circle, draw, fill=white, inner sep=3pt},
			every label/.style={scale=1, black}]		
			\node[vertex, label=below:$v_1^3$] (v1) at (0, 0) {};
			\node[vertex, label=below:$v_2^3$] (v2) at (2, 0) {};
			\node[vertex, label=left:$v_3^3$] (v3) at (1, 1) {};		
			\draw[line width=0.5mm, dashed] (v1) -- (v2);
			\draw[line width=0.5mm, darkgreen] (v1) -- (v3);
			\draw[line width=0.5mm, darkgreen] (v2) -- (v3);
			
		\end{tikzpicture}
		\caption{$H_3$}
		\end{subfigure}	
		\caption{Graphs Builder can create in 3 rounds, where $\{\text{dotted}, \text{dashed}\} = \{\text{red}, \text{blue}\}$.}
		\label{fig:c-unit2}
	\end{figure}
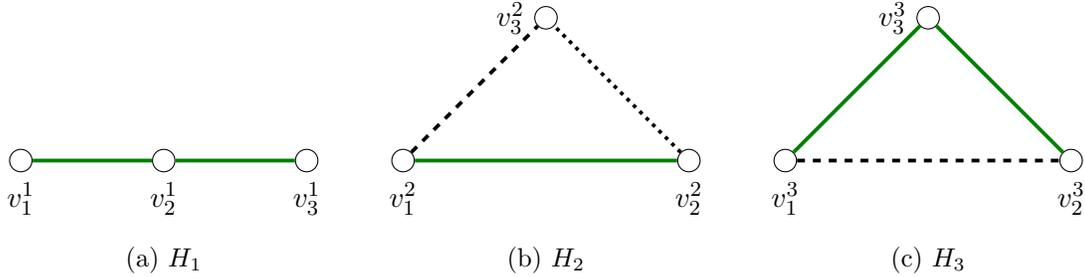

	\begin{proof}
	Note that the dashed edge in $H_3$ indicates that the edge can be red or blue. In $H_2$, the two dashed edges respectively represent one red edge and one blue edge. Also observe that the vertex labels in the graphs are introduced solely for the convenience of subsequent proofs and are independent of the order in which the edges are drawn.

	Builder first constructs a $P_3$. If all edges of this $P_3$ are green, then $H_1$ is obtained; otherwise, connecting the two endpoints of this $P_3$ yields either $H_2$ or $H_3$.
	\end{proof}

\begin{lemma}\label{mainthm2} 
Builder either constructs a green $P_6$ within 5 rounds, a good $P_6$ within 6 rounds, the graph $H$ shown in Figure~\ref{fig:H} within 7 rounds, a better $P_6$ within 8 rounds, or a best $P_6$ within 9 rounds.
\end{lemma}

		\begin{figure}[ht]
		\centering
		\definecolor{darkgreen}{rgb}{0.0, 0.5, 0.0} 
			\begin{subfigure}{0.3\textwidth}
			\centering
			\begin{tikzpicture}[scale=1.5, 
				vertex/.style={circle, draw, fill=white, inner sep=3pt},
				every label/.style={scale=1, black}]
				
				\node[vertex, label=below:$v_1$] (v1) at (0, 0) {};
					\node[vertex, label=left:$v_2$] (v2) at (0.5, 1) {};
				\node[vertex, label=below:$v_3$] (v3) at (1, 0) {};
				\node[vertex, label=below:$v_4$] (v4) at (2, 0) {};
				\node[vertex, label=left:$v_5$] (v5) at (2.5, 1) {};
				\node[vertex, label=below:$v_6$] (v6) at (3, 0) {};
				\draw[line width=0.5mm, darkgreen] (v1) -- (v2);
				\draw[line width=0.5mm, darkgreen] (v2) -- (v3);
				\draw[line width=0.5mm, darkgreen] (v3) -- (v4);
				\draw[line width=0.5mm, darkgreen] (v4) -- (v5);
				\draw[line width=0.5mm, darkgreen] (v5) -- (v6);
				\draw[line width=0.5mm, dashed] (v1) -- (v3);
				\draw[line width=0.5mm, dashed] (v4) -- (v6);
			
			\end{tikzpicture}
				\caption*{$H$}
				\end{subfigure}
			\caption{$H$ built by Builder in 7 rounds, where $\{\text{dashed}\} = \{\text{red}, \text{blue}\}$.}
			\label{fig:H}		
		\end{figure}
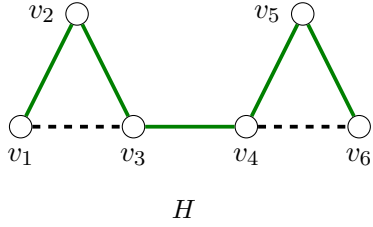

	\begin{proof} 
	Builder applies Lemma~\ref{cbasiclemma2} twice to construct two disjoint copies of the graphs shown in Figure~\ref{fig:c-unit2}, denoted by $(H_i,H_j)$. By symmetry, we may assume $1\leq i \leq j \leq 3$. We discuss all possible configurations of $(H_i,H_j)$ as follows.
	
	\setcounter{case}{0}
	\begin{case}
		$i=1$.
	\end{case}
	For the pair $(H_1,H_1)$, by Lemma~\ref{cmainlemma} (4), Builder either obtains a green $P_5$ within 5 rounds, a good $P_6$ within 6 rounds, or a better $P_6$ within 8 rounds.
	
	For the pair $(H_1,H_2)$, note that $H_2=C^{\text{RB}}_3$. By Lemma~\ref{Tlemma}, a better $P_6$ is obtained within 8 rounds.
	
	For the pair $(H_1,H_3)$, Builder draws the edge $v_3^1v_1^3$. If $v_3^1v_1^3$ is green, then a good $P_6$ is obtained in 6 rounds. If $v_3^1v_1^3$ is not green, then Builder draws the edge $v_1^1v_2^3$. If $v_1^1v_2^3$ is green, then a better $P_6$ is obtained in 7 rounds; otherwise, Builder draws the edge $v_3^1v_2^3$, which must be green, and thus a better $P_6$ is obtained in 8 rounds.

	\begin{case}
		$i=2$.
	\end{case}
	For the pair $(H_2,H_2)$, we relabel the vertices of one copy of $H_2$ as $u_1^2,u_2^2,u_3^2$. Without loss of generality, suppose $v_1^2v_3^2$ and $u_1^2u_3^2$ are red, while $v_2^2v_3^2$ and $u_2^2u_3^2$ are blue. Next, Builder draws the edges $v_2^2u_1^2$, $v_1^2u_3^2$, and $v_3^2u_2^2$, all of which must be green. Hence, within 9 rounds, the path $v_3^2u_2^2u_1^2v_2^2v_1^2u_3^2$ is formed, which is a best $P_6$.

	For the pair $(H_2,H_3)$, by Lemma~\ref{Tlemma}, a better $P_6$ is obtained within 8 rounds.

	\begin{case}
		$i=3$.
	\end{case}
	For the pair $(H_3,H_3)$, we relabel the vertices of one copy of $H_3$ as $u_1^3,u_2^3,u_3^3$. Builder draws the edge $u_2^3v_1^3$. If this edge is green, then the graph $H$ is obtained in 7 rounds. If $u_2^3v_1^3$ is not green, Builder then draws the edge $u_1^3v_1^3$, which must be green. Hence, a better $P_6$ is obtained within 8 rounds.
\end{proof}

\begin{lemma}\label{HLemma} 
Let $H$ be the graph shown in Figure~\ref{fig:H}, and let $P_n$ be a green path with $n\ge 2$ that is vertex-disjoint from $H$. Then within the next 4 rounds, Builder can construct a green $C_{n+6}$.
\end{lemma}
\begin{proof}
Let $x$ and $y$ be the endpoints of the green $P_n$. Builder draws the edges $xv_6$ and $yv_1$. 

If both $xv_6$ and $yv_1$ are green, then a green $C_{n+6}$ is obtained in 2 rounds. 

If exactly one of $xv_6$ and $yv_1$ is green, without loss of generality, assume $yv_1$ is not green. Then Builder draws the edges $yv_3$ and $v_1v_4$, both of which must be green. Consequently, within 4 rounds, a green $C_{n+6}$ is formed: $xP_n yv_3v_2v_1v_4v_5v_6x$.

If neither $xv_6$ nor $yv_1$ is green, then both $v_1$ and $v_6$ become better vertices. Builder then draws the edges $xv_1$ and $yv_6$, both of which must be green. Hence, within 4 rounds, a green $C_{n+6}$ is obtained: $xP_n yv_6v_5v_4v_3v_2v_1x$.
\end{proof}

\begin{theorem}\label{c4k-2}
For all $k\ge 4$, we have $\tilde{r}(P_3,P_3,C_{4k-2}) \leq 6k-2$.
\end{theorem}

\begin{proof}
By Theorem~\ref{finallemma}, Builder can either construct a green $P_{4k-8}$ within $6k-14$ rounds, or a good $P_{4k-8}$ within $6k-13$ rounds. Then, applying Lemma~\ref{mainthm2}, Builder constructs a green $P_6$ that is vertex-disjoint from $P_{4k-8}$.

If Builder obtains a green $P_{4k-8}$ within $6k-14$ rounds, and the $P_6$ is a green path constructed in 5 rounds, a good path within 6 rounds, or a better path within 8 rounds, then by applying parts (3), (2), and (4) of Lemma~\ref{Connection Lemma} respectively, we can connect $P_6$ with $P_{4k-8}$ to form a green $C_{4k-2}$ within a total of $6k-2$ rounds.

If Builder obtains a good $P_{4k-8}$ within $6k-13$ rounds, and the $P_6$ is a green path constructed in 5 rounds, a good path within 6 rounds, or a better path within 8 rounds, then by applying parts (2), (1), and (4) of Lemma~\ref{Connection Lemma} respectively, we can again obtain a green $C_{4k-2}$ within $6k-2$ rounds in total.

If the $P_6$ is a best path constructed within 9 rounds, then it can be connected to $P_{4k-8}$ in 2 additional rounds, resulting in a green $C_{4k-2}$ within $6k-2$ rounds in total.

If the $P_6$ is the graph $H$ depicted in Figure~\ref{fig:H} and is constructed within 7 rounds, then by Lemma~\ref{HLemma}, it can be connected to $P_{4k-8}$ in 4 additional rounds to form a green $C_{4k-2}$, still within a total of $6k-2$ rounds.
\end{proof}

\begin{lemma}\label{c4k+3Lemma}
	When $k\ge 4$, Builder either constructs a green $P_{4k-5}$ within $6k-11$ rounds, or a good $P_{4k-5}$ within $6k-9$ rounds, or a better $P_{4k-5}$ within $6k-8$ rounds.
\end{lemma}
\begin{proof}
	We first apply Theorem~\ref{finallemma}, by which Builder either obtains a green $P_{4k-8}$ within $6k-14$ rounds, or a good $P_{4k-8}$ within $6k-13$ rounds. Then Builder applies Lemma~\ref{cbasiclemma2} to construct one of the graphs $H_1$, $H_2$, or $H_3$, each vertex-disjoint from the $P_{4k-8}$.

	If Builder obtains a green $P_{4k-8}$ within $6k-14$ rounds, we apply Lemma~\ref{cmainlemma} (4) in the case of $H_1$, Lemma~\ref{Tlemma} in the case of $H_2$, and Lemma~\ref{cmainlemma} (2) in the case of $H_3$, which proves the lemma.

	If Builder obtains a good $P_{4k-8}$ within $6k-13$ rounds, then we apply Lemma~\ref{cmainlemma} (2) for $H_1$, and Lemma~\ref{Tlemma} for $H_2$, which proves the lemma.

	If Builder constructs a good $P_{4k-8}$ and a copy of $H_3$ disjoint from it, let $x$ denote a good endpoint of the $P_{4k-8}$ and $y$ its other endpoint. Builder adds the edge $yv_1^3$. If $yv_1^3$ is green, then a good $P_{4k-5}$ is obtained. If $yv_1^3$ is not green, then $v_1^3$ becomes a better point. In this case, Builder draws the necessarily green edge $yv_2^3$, resulting in a better $P_{4k-5}$: $xP_{4k-8}yv_2^3v_3^3v_1^3$.
\end{proof}

\begin{theorem}\label{c4k+3}
For all $k\ge 4$, we have $\tilde{r}(P_3,P_3,C_{4k+3}) \leq 6k+5$.
\end{theorem}

\begin{proof}
We first apply Lemma~\ref{c4k+3Lemma}, by which Builder constructs a green $P_{4k-5}$. Then, using Lemma~\ref{mainthm}, Builder obtains a graph that is vertex-disjoint from $P_{4k-5}$. This graph is either $G_1\cup G_1$ (where $G_1$ is shown in Figure~\ref{fig:c-unit}(a)) or a green $P_8$. We proceed by considering the following three cases.

\setcounter{case}{0}
\begin{case}
$P_{4k-5}$ is a green path constructed within $6k-11$ rounds.
\end{case}

If $P_8$ is a green path constructed in 7 rounds, a good path in 9 rounds, or a better path in 11 rounds, then by applying parts (3), (2), and (4) of Lemma~\ref{Connection Lemma}, respectively, we can connect $P_8$ and $P_{4k-5}$ to obtain a green $C_{4k+3}$. If Builder constructs $G_1 \cup G_1$ using Lemma~\ref{mainthm}, then Lemma~\ref{G1andG1} guarantees that a better $P_8$ can be formed in the following 6 rounds. Finally, applying Lemma~\ref{Connection Lemma} (4) yields a green $C_{4k+3}$ within a total of $6k+5$ rounds.

\begin{case}
$P_{4k-5}$ is a good path constructed within $6k-9$ rounds.
\end{case}

If $P_8$ is a green path constructed in 7 rounds, a good path in 9 rounds, or a better path in 11 rounds, then applying parts (2), (1), and (4) of Lemma~\ref{Connection Lemma}, respectively, connects $P_8$ with $P_{4k-5}$ to yield a green $C_{4k+3}$. If Builder constructs $G_1 \cup G_1$ via Lemma~\ref{mainthm}, then applying Lemma~\ref{G1andG1}, a better $P_8$ can be constructed in the following 6 rounds. Moreover, this better $P_8$ can be constructed so that the color of the edge incident to its good endpoint is opposite to that of the edge incident to the good endpoint of $P_{4k-5}$. It follows that in 2 additional rounds, the good $P_{4k-5}$ and the better $P_8$ can be connected to form a green $C_{4k+3}$. Thus, the total number of rounds does not exceed $6k+5$.

\begin{case}
$P_{4k-5}$ is a better path constructed within $6k-8$ rounds.
\end{case}

Let the endpoints of $P_{4k-5}$ be $x$ and $y$, where $x$ is a good vertex and $y$ is a better vertex. If $P_8$ is a green path constructed in 7 rounds or a good path in 9 rounds, then applying Lemma~\ref{Connection Lemma} (4) yields a green $C_{4k+3}$. If $P_8$ is a better path constructed in 11 rounds, then 2 more rounds suffice to obtain a green $C_{4k+3}$. If Builder constructs $G_1 \cup G_1$ via Lemma~\ref{mainthm}, then for convenience, we relabel the vertices $v_i^1$ in one of the $G_1$ graphs (as shown in Figure~\ref{fig:c-unit}) as $u_i^1$ for $i\in [4]$. Builder draws the following edges, which must be green: $v_3^1u_2^1$, $u_2^1u_3^1$, $u_3^1v_2^1$, $v_2^1u_4^1$, $u_1^1v_4^1$. This forms, in 5 rounds, a good $P_8$ whose good vertex is incident to a red edge. By symmetry, Builder can also obtain a good $P_8$ with the good endpoint incident to a blue edge in 5 rounds. Therefore, without loss of generality, we may assume that $x$ is incident to a blue edge, and in 2 further rounds, a green $C_{4k+3}$ can be formed with the better $P_{4k-5}$. In all of the above scenarios, the total number of rounds is at most $6k+5$.
\end{proof}

Combining Theorems~\ref{c4k}, \ref{c4k+1}, \ref{c4k-2}, and \ref{c4k+3}, we obtain the upper bound in Theorem~\ref{thm:main2}.\qed

\end{document}